\documentclass[11pt]{scrartcl}
\RequirePackage{enumerate}
\RequirePackage{verbatim}
\RequirePackage{hyperref}

\RequirePackage{amsmath,amsfonts,amssymb}
\RequirePackage[mathcal,mathbf]{euler}

\usepackage{latexsym,amscd,amsthm}
\usepackage[all]{xy}

\usepackage[dvips]{epsfig}
\usepackage[dvips]{graphics}
\usepackage{lscape}

\input cyracc.def

\makeatletter
\def\operator@font{\sf}
\makeatother

\newcommand{\cA}{{\mathscr A}}
\newcommand{\cB}{{\mathscr B}}

\newcommand{\cC}{{\mathscr C}}

\newcommand{\cF}{{\mathscr F}}

\newcommand{\cI}{{\mathscr I}}

\newcommand{\cO}{{\mathscr O}}

\newcommand{\cU}{{\mathscr U}}

\newcommand{\Del}{{\mathsf{Del}}}
\newcommand{\TW}{{\mathsf{TW}}}
\newcommand{\Tot}{{\mathsf{Tot}}}

\newcommand{\ad}{{\mathsf{ad}}}

\newcommand{\eq}{{\mathsf{eq}}}

\newcommand{\HKR}{{\mathsf{HKR}}}

\newcommand{\PBW}{{\mathsf{PBW}}}

\newcommand{\gog}{{\mathfrak{g}}}

\newcommand{\MC}{\mathsf{MC}}

\newcommand{\sHom}{\underline{\mathsf{Hom}}}

\newcommand{\D}{{\bf D}_{\mathsf{coh}}^b}

\newcommand{\CE}{{\mathsf{CE}}}

\newcommand{\chk}{{\scriptscriptstyle\vee}}

\newcommand{\pt}{{\mathsf{pt}}}

\DeclareMathOperator{\Spec}{Spec}

\DeclareMathOperator{\Der}{Der}

\DeclareMathOperator{\Aut}{Aut}

\DeclareMathOperator{\End}{End}
\DeclareMathOperator{\Hom}{Hom}
\DeclareMathOperator{\Tor}{Tor}

\DeclareMathOperator{\id}{id}

\DeclareMathOperator{\Ext}{Ext}

\newcommand{\bS}{\mathbf{S}}

\newcommand{\U}{\mathbf{U}}

\DeclareMathOperator{\sym}{sym}

\newcommand{\ra}{\rightarrow}

\newcommand*{\triple}[2][.1ex]{%
  \mathrel{\vcenter{\offinterlineskip%
  \hbox{$#2$}\vskip#1\hbox{$#2$}\vskip#1\hbox{$#2$}}}}

\newcommand*{\triplearrows}{\triple{\rightarrow}}

\newtheorem{thm}{Theorem}[section]

\newtheorem{lem}[thm]{Lemma}
\newtheorem{cor}[thm]{Corollary}
\newtheorem{prop}[thm]{Proposition}

\theoremstyle{definition}
\newtheorem{defi}[thm]{Definition}
\newtheorem{rem}[thm]{Remark}
\newtheorem{exa}[thm]{Example}

\renewcommand{\phi}{\varphi}

\author{
Damien Calaque\footnote{Department of Mathematics, ETH Z\"urich, 8092 Z\"urich, Switzerland, {\em e-mail:} {\tt damien.calaque@math.ethz.ch}. 
On leave from Institut Camille Jordan, CNRS \& Universit\'e Lyon 1, France.}, 
Andrei C\u ald\u araru\footnote{Department of Mathematics, University of Wisconsin--Madison, 480 Lincoln Drive, Madison WI 53706-1388, USA, {\em e-mail:} {\tt andreic@math.wisc.edu}.}, 
Junwu Tu\footnote{Mathematics Department, University of Oregon, Eugene OR 97403, USA, {\em e-mail:} {\tt junwut@uoregon.edu}.}}

\title{On the Lie algebroid of a derived self-intersection}
\date{}

\makeatletter
\newcommand{\wtd}{\widetilde}

\newcommand*{\doublerightarrow}[2]{\mathrel{
  \settowidth{\@tempdima}{$\scriptstyle#1$}
  \settowidth{\@tempdimb}{$\scriptstyle#2$}
  \ifdim\@tempdimb>\@tempdima \@tempdima=\@tempdimb\fi
  \mathop{\vcenter{
    \offinterlineskip\ialign{\hbox to\dimexpr\@tempdima+1em{##}\cr
    \rightarrowfill\cr\noalign{\kern.01ex}
    \rightarrowfill\cr}}}\limits^{\!#1}_{\!#2}}}
\newcommand*{\triplerightarrow}[1]{\mathrel{
  \settowidth{\@tempdima}{$\scriptstyle#1$}
  \mathop{\vcenter{
    \offinterlineskip\ialign{\hbox to\dimexpr\@tempdima+1em{##}\cr
    \rightarrowfill\cr\noalign{\kern.01ex}
    \rightarrowfill\cr\noalign{\kern.01ex}
    \rightarrowfill\cr}}}\limits^{\!#1}}}

\newcommand{\Rnum}[1]{\expandafter\@slowromancap\romannumeral #1@}
\makeatother
\numberwithin{equation}{section}

\begin{document}
\maketitle

\begin{abstract}
\noindent{\bf Abstract.} 
Let $i:X\hookrightarrow Y$ be a closed embedding of smooth algebraic varieties. 
Denote by $N$ the normal bundle of $X$ in $Y$. 
The present paper contains two constructions of certain Lie structure on the shifted normal bundle $N[-1]$ encoding the information of the formal neighborhood of $X$ in $Y$. 
We also present a few applications of these Lie theoretic constructions in understanding the algebraic geometry of embeddings.
\end{abstract}

\setcounter{tocdepth}{1}
\tableofcontents
%\newpage
\section{Introduction}\label{sec:intro}

The aim of this paper is to study the derived self-intersection of a closed $X$ into $Y$, when $X$ and $Y$ are smooth algebraic varieties. 
The word ``derived'' appears here because self-intersections are badly-behaved (e.g.~they are not transverse and their actual dimension does not coincide with the expected one). 
If $X$ and $Y$ were to be differentiable manifolds we could consider an appropriate small perturbation $X_\epsilon$ of $X$ and take $X\cap X_\epsilon$, but this approach has several drawbacks: \\[-0.7cm]
\begin{itemize}
\item it does not preserve the set-theoretical intersection; \\[-0.7cm]
\item it does lead to some category theoretic problems (no functorial choice for $X_\epsilon$); \nopagebreak \\[-0.7cm]
\item we can't do this with algebraic varieties. 
\end{itemize}
It has been known for a long time that a possible replacement for actual geometric perturbation is homological perturbation. 
More percisely, if $i:X\hookrightarrow Y$ is a closed embedding then we will consider a resolution $\mathcal R$ of $\mathcal O_X$ by a differential graded (dg) $i^{-1}\mathcal O_Y$-algebra, and equip $X$ with the sheaf of dg-rings 
$\mathcal O_X\otimes^{\mathbb{L}}_{i^{-1}\mathcal O_Y}\mathcal O_X:=\mathcal R\otimes_{i^{-1}\mathcal O_Y}\mathcal O_X$. 
The pair $(X,\mathcal O_X\otimes^{\mathbb{L}}_{i^{-1}\mathcal O_Y}\mathcal O_X)$ is a dg-scheme that we will denote $X\times_Y^hX$ and call the derived self-intersection of $X$ into $Y$ (or more generally the self homotopy fiber product of $X$ over $Y$ for a general morphism $X\to Y$). 

Observe that different choices of resolutions give rise to weakly equivalent results, and that there exists functorial choices for such resolutions. 

\medskip

In the rest of the introduction we allow ourselves to deal informally with schemes as if they would be topological 
spaces, where resolutions have to be thought as fibrant replacements\footnote{This can be made precise using model 
categories and the advanced technology of homotopical and derived algebraic geometry, after To\"en-Vezzosi \cite{TV1,TV2} 
and Lurie \cite{L}. Despite the fact that this is certainly the appropriate framework to work with, we will 
nevertheless stay within the realm of dg-schemes (after \cite{CK}), which are sufficient for our purposes. }. 

\subsection{The diagonal embedding (after M.~Kapranov)}

Let us consider the example of a diagonal embedding $\Delta:X\hookrightarrow X\times X$. One way to compute the 
homotopy fiber product is to factor $\Delta:X\to X\times X$ into an acyclic cofibration followed by a fibration: 
$X\tilde\rightarrowtail \tilde X\twoheadrightarrow X\times X$. 
This is achieved by taking $\tilde X:=\{\gamma:[0,1]\to X\times X\,|\,\gamma(0)\in\Delta(X)\}$ and 
$\gamma\mapsto\gamma(1)$ as fibration. Therefore the derived self-intersection of the diagonal is 
$$
\tilde{X}\times_{X\times X}X=\{\gamma:[0,1]\to X\times X\,|\,\gamma(0),\gamma(1)\in\Delta(X)\}\,,
$$
which happens to be weakly equivalent to the loop space of $X$ ({\it via} the map sending a path $\gamma$ in 
$X\times X$ with both ends in the diagonal to the loop $\tilde\gamma$ in $X$ defined by 
$\tilde\gamma(t)=\begin{cases}\pi_1\big(\gamma(2t)\big)\textrm{ if }0\leq t\leq 1/2 \\ 
\pi_2\big(\gamma(2-2t)\big)\textrm{ if }1/2\leq t\leq 1\end{cases}$). 

\medskip 

By analogy with what happens to topological spaces we get that $X\times^h_{X\times X}X$ is a derived group 
$X$-scheme\footnote{It is an $X$-scheme because $X\times X$ itself is an $X$-scheme through one of the projections 
$X\times X\to X$. Moreover, the choice of one or the other is ``homotopically irrelevant'' in our context. 
Notice nevertheless that it would be a bit more subtle if we were looking at the $2$-groupoid structure 
on $X\times^h_{X\times X}X$. }. 
%\begin{question}
%What is the Lie algebra of $X\times^h_{X\times X}X$?
%\end{question}

\subsubsection*{What is the Lie algebra of $X\times^h_{X\times X}X$ ?}

The answer to this question is contained in a beautiful paper of Kapranov \cite{Ka}: informally speaking this is 
the Atiyah extension $S^2(T_X)\to T_X[1]$, viewed as a Lie bracket $\wedge^2(T_X[-1])\to T_X[-1]$ on $T_X[-1]$ 
in ${\bf D}(X)$. Let us mention that this is the induced Lie bracket, on cohomology, of the differential graded Lie algebra structure 
on $\mathfrak{g}_X:=\mathbb{T}_{X/X\times X}$. We should therefore obtain, by homotopy transfer, higher 
(Massey) brackets on the cohomology $T_X[-1]$, which were also constructed by other means in Kapranov's paper. 
%\begin{question}
%What is the universal enveloping algebra of this Lie algebra?
%\end{question}

\subsubsection*{What is the universal enveloping algebra of this Lie algebra ?}

The answer to this question was given by Markarian \cite{Ma}: this is the Hochschild cochain complex 
$\mathcal H\mathcal H_X:=(\pi_1)_*\mathbb{R}\mathcal Hom_{X\times X}(\Delta_*\mathcal O_X,\Delta_*\mathcal O_X)$ of $X$. 
The PBW theorem then reads
$$
S\big(T_X[-1]\big)\,\tilde\longrightarrow\,\mathcal H\mathcal H_X\quad\big(\textrm{in }{\bf D}(X)\big)\,,
$$
which is nothing but the HKR theorem \cite{HKR,Ye}. 

Observe moreover that any object $E$ in ${\bf D}(X)$ is naturally a representation over the Lie algebra 
$T_X[-1]$, via its own Atiyah class, and that morphisms in ${\bf D}(X)$ are all $T_X[-1]$-linear. 
The $T_X[-1]$-invariant space of an object $E$ is then the space of morphisms from $\mathcal O_X$, 
the trivial representation, to $E$: these are (derived) global sections. Assuming that the Duflo 
isomorphism\footnote{Which states that for a finite dimensional Lie algebra $\mathfrak{g}$, 
$S(\mathfrak{g})^{\mathfrak{g}}$ and $U(\mathfrak{g})^{\mathfrak{g}}$ are isomorphic as algebras. } 
is valid in dg or triangulated categories (which is unknown), we would then have an isomorphism of algebras 
$$
H^*\Big(X,S\big(T_X[-1]\big)\Big)\,\tilde\longrightarrow\,{\rm HH}^*_X
:=Ext^*_{X\times X}(\Delta_*\mathcal O_X,\Delta_*\mathcal O_X)\,.
$$
This result has been proved using a different approach in \cite{CVdB,DTT}. 

%\begin{question}
%What is the Chevalley-Eilenberg algebra of $\mathfrak{g}_X$?
%\end{question}

\subsubsection*{What is the Chevalley-Eilenberg algebra of $\mathfrak{g}_X$ ?}

The answer to this question is again in Kapranov's paper: $C^*(\mathfrak g_{X})$ is quasi-isomorphic to 
the structure ring $\mathcal O_{X_{X\times X}^\infty}$ of the formal neighborhood $X_{X\times X}^\infty$ of the 
diagonal into $X\times X$. 

\subsection{The general case of a closed embedding (this paper)}

For a general closed embedding $i:X\hookrightarrow Y$ we again resolve it by considering a path space 
$$
\tilde{X}:=\{\gamma:[0,1]\to Y\,|\,\gamma(0)\in i(X)\}
$$
and $\gamma\mapsto\gamma(1)$ as fibration. Therefore the derived self-intersection of $X$ into $Y$ is 
$$
\tilde{X}\times_{Y}X=\{\gamma:[0,1]\to Y\,|\,\gamma(0),\gamma(1)\in i(X)\}\,,
$$
which this time does not appear to be a group $X$-scheme. Nevertheless $X\times^h_YX$ happens to be 
the space of arrows of a derived groupoid $Y$-scheme having $X$ as space of objects. 
%\begin{question}
%What is the Lie algebroid of $X\times^h_YX$?
%\end{question}

\subsubsection*{What is the Lie algebroid of $X\times^h_YX$ ?}

We answer that, informally speaking, the Lie algebroid of the groupoid scheme $X\times^h_YX$ is 
the shifted normal bundle $N[-1]$ of $X$, with anchor map being the Kodaira-Spencer class $N[-1]\to T_X$. 
We observe again that this is the induced anchor map, on cohomology, of the base change morphism 
$\mathbb{T}_{X/Y}\to \mathbb{T}_{X/\pt}$. To state this observation precisely, recall from~\cite[Section 2.6]{CK} that over the scheme $Y$ we may resolve $i_*\mathcal O_X$ by a sheaf of $\mathcal O_Y$-algebra of the form $\bS_{\cO_Y}(E[1])$ where $E=E_0\oplus E_{-1} \oplus E_{-2}\cdots$ is a non-positively graded $\cO_Y$-locally free sheaf such that each graded component $E_j$ is of finite rank. We denote the differential graded ringed space $\big(Y, \bS_{\cO_Y}(E[1])\big)$ by $\widetilde{X}$. By construction, there are morphisms of dg ringed spaces
\[ X\overset{j}{\to}\widetilde{X}, \mbox{\; and \;} \widetilde{X} \overset{\pi}{\to}Y\]
with composition $\pi\circ j=i$. We call this data a smooth resolution of the morphism $i$.
\begin{prop}
Let $X\overset{j}{\to}\widetilde{X}\overset{\pi}{\to}Y$ be a smooth resolution of the inclusion morphism $i:X\hookrightarrow Y$. 
Then the pair $(\mathcal O_{\widetilde{X}},T_{\widetilde{X}/Y})$ is a dg-Lie algebroid over $\mathcal O_Y$ and 
it is the dg-Lie algebroid of the following groupoid in dg-$Y$-schemes: $(\widetilde{X},\widetilde{X}\times_Y\widetilde{X})$. 
\end{prop}
Recall that $j^*:\D(\widetilde{X})\to\D(X)$ is a triangulated equivalence and that 
$\mathbb{T}_{X/Y}:=j^*T_{\widetilde{X}/Y}\cong N[-1]$ in $\D(X)$. 
In particular, we get from the above that $(i_*\mathcal O_X,i_*N[-1])$ is a Lie algebroid in $\D(X)$. 

\subsubsection*{What is the universal enveloping algebra of this Lie algebroid ?}

The enveloping algebra of a Lie algebroid with base $X$ has to be an $\mathcal O_{X\times X}$-module set-theoretically supported on the 
diagonal. We answer that the universal enveloping algebra of the Lie algebroid of $X\times^h_YX$ is the kernel $K$ of $i^*i_\dagger=(i^*i_*)^\vee$, 
where the functor $i_\dagger$ is the left adjoint of $i^*$ (see Subsection~\ref{sec-2.5} for a formula of $i_\dagger$). The algebra structure $K\circ K\to K$ is given by the natural transformation $i^*i_\dagger i^*i_\dagger \implies i^*i_\dagger $. 
More precisely (see Subsection~\ref{sec-2.5} and Section~\ref{sec:mon}): 
\begin{thm}\label{thm:intro}
Let $X\overset{j}{\to}\widetilde{X}\overset{\pi}{\to}Y$ be a smooth resolution of the inclusion morphism $i:X\hookrightarrow Y$. 
Then we have a natural quasi-isomorphism of dg-algebras 
$\U(T_{\widetilde{X}/Y})\longrightarrow\mathcal{H}om_{\mathcal O_Y}(\mathcal O_{\widetilde{X}},\mathcal O_{\widetilde{X}})$. 
In other words, $\U(T_{\widetilde{X}/Y})$ represents the monad $i^*i_\dagger $. 
\end{thm}
One can actually prove that the above identifications are somehow ``compatible with coproducts''. 
This can be seen {\it via} the following dual statement: 
\begin{thm}\label{thm:intro2}
Let $X\overset{j}{\to}\widetilde{X}\overset{\pi}{\to}Y$ be a smooth resolution of the inclusion morphism $i:X\hookrightarrow Y$. Let ${\bf J} (T_{\widetilde{X}/Y})$ be the infinite jet algebra of the Lie algebroid $T_{\widetilde{X}/Y}$.
Then we have a natural quasi-isomorphism of dg-algebras 
${\bf J}(T_{\widetilde{X}/Y})\longrightarrow\mathcal{O}_{\widetilde{X}\times_Y\widetilde{X}}$. 
In other words, ${\bf J}(T_{\widetilde{X}/Y})$ represents the monad $i^*i_*$. 
\end{thm}

\subsubsection*{What is the Chevalley-Eilenberg algebra of this Lie algebroid ?}

We answer that this is quasi-isomorphic to the structure ring $\mathcal O_{X_{Y}^\infty}$ of the formal neighborhood 
$X_{Y}^\infty$ of the embedding $i:X\hookrightarrow Y$, generalizing Kapranov's result in the case diagonal embeddings.
See Theorem \ref{thm:ce}. 

\medskip

By a kind of homotopy transfer procedure one should get a minimal $L_\infty$-algebroid structure on $N[-1]$ 
(see Definition \ref{def:lie}). 
We actually don't know how to get homotopy transfer for sheaves, but are able to prove the following 
(it appears as Proposition \ref{coro:main} in Section \ref{sec:homotopy}): 
\begin{prop}
There exists a minimal $L_\infty$-algebroid structure on $N[-1]$ whose Chevalley-Eilenberg algebra is quasi-isomorphic 
$\mathcal O_{X_{Y}^\infty}$. 
\end{prop}
We obtain that way ``higher anchor maps'' %$\rho_k:S^k(N)\to T_X[1]$ ($k\geq2$)
that can be interpreted as succesive obstructions (which already appeared in \cite{ABT}) against being able to split 
the embedding of $X$ into its $l$-th infinitesimal neighborhood $X^{(l)}_Y$. We also obtain ``higher 
brackets'' which are interpreted as succesive obstructions (which also appeared in \cite{ABT}) against linearizing $X^{(l)}_Y$. 
%I.e. having a $l$-th tubular neighborhood: $X^{(l)}_Y$ isomorphic to $X^{(l)}_N$. 

\subsection{Plan of the paper}

We start in Section~\ref{sec:dg} by describing the dg-Lie algebroid associated with a closed embedding. We first recall quickly some 
known facts about the (co)tangent complex and about Lie algebroids, with some emphasis on relative derivations. 
The Lie algebroid structure on the relative tangent complex of a morphism $f:X\to Y$ then becomes obvious. 
We finally prove that, in the case of a closed embedding $i:X\hookrightarrow Y$, the Chevalley-Eilenberg algebra of this Lie algebroid 
is quasi-isomorphic to the structure sheaf of $X_Y^{(\infty)}$, its jet algebra represents $i^*i_*$ and its universal enveloping algebra 
represents $i^*i_\dagger $. 

In Section~\ref{sec:mon} we show that the above identifications are compatible with natural Hopf-like structures on these objects. 

In Section~\ref{sec:deform} we review about cosimplicial methods and a correspondence betwen Maurer-Cartan elements and non-abelian $1$-cocycles. 

Section~\ref{sec:homotopy} is devoted to the construction of an $L_\infty$-algebroid structure on the shifted normal bundle $N[-1]$ of $X$ into $Y$, 
which is such that its Chevalley-Eilenberg algebra is also quasi-isomorphic to the structure sheaf of $X_Y^{(\infty)}$. 
We then prove that the structure maps of this $L_\infty$-algebroid structure provides obstructions to splitting the inclusion $X\to X_Y^{(k)}$ 
and linearizing $X_Y^{(k)}$. This allows us to recover recent results of \cite{ABT}.

\subsection{Acknowledgements}

Shilin Yu \cite{Yu} independently obtained similar results for complex manifolds (more in the spirit of Kapranov's original formulation). 
We warmly thank him for the numerous enlightening discussions we had before and during the writing of the present paper. We are also grateful to Alexander Polishchuk from whom we learned the proof of Lemma~\ref{local}.
%D.C.~also thanks Luc Illusie for pointing reference \cite{G} to him. 
The work of D.C. is supported by the Swiss National Science Foundation (grant number $200021\underline{~}137778$). 

\subsection{Notation}

We work over an algebraically closed field ${\bf k}$ of characteristic zero. 
Unless otherwise specified all (dg-)algebras, (dg-)schemes, varieties, etc... are over ${\bf k}$. 

\medskip

\noindent We often denote by $\pt$ the terminal dg-scheme. Namely, $\pt:=\Spec({\bf k})$. 

\medskip

\noindent Whenever we have a dg-algebra $A$ we denote by $A^\sharp$ the underlying graded algebra and by $d_A$ its differential. 
Then $A=(A^\sharp,d_A)$. 
Similarly, given a dg-scheme $X$ we denote by $X_\sharp$ the dg-scheme having the same underlying scheme and structure sheaf $\cO_{X_\sharp}:=\cO_{X}^\sharp$. 

\medskip

\noindent For a graded module $E$ and an integer $k$, we denote by $E[k]$ the graded module whose $l$-th graded piece is the $k+l$-th graded piece of $E$. The obious maps $E\longrightarrow E[-1]$ and $E\longrightarrow E[1]$ are respectively denoted by $s$ and $e\mapsto \bar{e}$. 

\medskip

\noindent We freely extend the above notation to sheaves. 

\medskip

\noindent By a Lie algebroid we mean a sheaf of Lie-Rinehart algebras (L-R algebras in \cite{R}, to which we refer for more details). 

%\newpage
\section{The dg-Lie algebroid of a closed embedding}\label{sec:dg}

In this section we construct a differential graded Lie algebroid associated to a closed embedding $i:X\hookrightarrow Y$ of smooth algebraic varieties. 
We show that the Chevalley-Eilenberg algebra of this differential graded Lie algebroid is quasi-isomorphic to the formal neighborhood algebra of $X$ in $Y$; while the dual of its universal enveloping algebra is quasi-isomorphic to the formal neighborhood of $X$ inside the derived self-intersection $X\times^h_Y X$. 
We shall work with Ciocan-Fontanine and Kapranov's dg-schemes~\cite[Section 2]{CK}.

\subsection{Short review of the relative (co)tangent complex}

Let $i:X\hookrightarrow Y$ be a closed embedding of smooth algebraic varieties with normal bundle $N$. 
Let us also assume that $Y$ is quasi-projective to ensure existence of resolutions by locally free sheaves on $Y$. 
By the constructions of~\cite[Theorem 2.7.6]{CK} the map $i$ can be factored as
$$
X \overset{j}{\longrightarrow} \widetilde{X} \overset{\pi}{\longrightarrow} Y\,,
$$
where $j$ is a quasi-isomorphic closed embedding and $\pi$ is smooth. 
Moreover since we are assuming $Y$ is quasi-projective and smooth, the ordinary scheme underlying the dg-scheme $\widetilde{X}$ can in fact be taken to be just $Y$. 
Its structure sheaf is of the form 
$$
\cO_{\widetilde{X}}:=\big(\bS_{\cO_Y}(E[1]),Q\big)
$$
for some non-positively graded locally free sheaf $E$ on $Y$, and $Q$ is a degree one $\cO_Y$-linear derivation on $\bS_{\cO_Y}(E[1])$ that squares to zero. 
The appearance of a shift on $E$ is to mimic the case when $X$ is a complete intersection in $Y$. 
By construction there is a quasi-isomorphism of dg-$\mathcal O_Y$-algebras
$$
\cO_{\widetilde{X}} \longrightarrow j_*\cO_X\,.
$$
The relative cotangent complex $\mathbb{L}_{X/Y}$ is by definition 
$$
\mathbb{L}_{X/Y}:=j^*L_{\wtd{X}/Y}\,,\quad\textrm{where}\quad
L_{\wtd{X}/Y}:=\big(\Omega^1_{\cO_{\wtd{X}}^\sharp/\cO_Y},L_Q\big)\,.
$$
The differential $L_Q$ is given by the Lie derivation action on the graded space of $1$-forms $\Omega^1_{\cO_{\wtd{X}}^\sharp/\cO_Y}$. 
Since $Q$ is an odd derivation that squares to zero, its Lie derivative $L_Q$ also squares to zero because $L_Q^2=\frac{1}{2}L_{[Q,Q]}=0$. 
It is well-defined up to quasi-isomorphism (see \cite[Proposition 2.7.7]{CK} for a precise statement). 
Moreover, the restriction morphism $L_{\wtd{X}/Y}\longrightarrow j_*j^*L_{\wtd{X}/Y}=j_*\mathbb{L}_{X/Y}$ is a quasi-isomorphism. 
In particular this implies that the dg-sheaf $L_{\wtd{X}/Y}$ is isomorphic to $j_*N^\chk[1]$ as an object in $\D(\wtd{X})$. 
Namely, we have two distinguished triangles 
$$
i^*\Omega^1_{Y/\pt}\longrightarrow\Omega^1_{X/\pt}\longrightarrow \mathbb{L}_{X/Y}\overset{+1}{\longrightarrow}
\qquad\textrm{and}\qquad
N^\chk\longrightarrow i^*\Omega^1_{Y/\pt}\longrightarrow \Omega^1_{X/\pt}\overset{+1}{\longrightarrow}\,.
$$

Dually we can define the tangent complex $\mathbb{T}_{X/Y}$ as $j^*T_{\wtd{X}/Y}$, where $T_{\wtd{X}/Y}:=\Der_{\cO_Y}(\cO_{\wtd{X}})$ is endowed with the differential $[Q,-]$. 
The fact that this differential squares to zero follows from the graded version of Jacobi identity for vector fields. 
The dg-sheaf $T_{\wtd{X}/Y}$ is isomorphic to $i_*N[-1]$ in $\D(Y)$\footnote{We are making here a slight abuse of language. 
Strictly speaking, $T_{\wtd{X}/Y}$ is quasi-isomorphic to $j_*N[-1]$ as an $\mathcal O_{\wtd{X}}$-module and thus $\pi_*T_{\wtd{X}/Y}$ 
is quasi-isomorphic to $i_*N[-1]$ as an $\mathcal O_Y$-module. But for an $\mathcal O_{\wtd{X}}$-module $F$, $\pi_*F$ is simply 
the same sheaf (on $Y$) viewed as an $\mathcal O_Y$-module in the obvious way. We will therefore allow ourselves to omit $\pi_*$ from the notation when the context is clear enough. }. 

\subsection{Recollection on (dg-)Lie algebroids}

There is a natural Lie bracket on $T_{\wtd{X}/Y}$, being the space of $\cO_Y$-linear derivations of $\cO_{\wtd{X}}$. 
More precisely, the pair $(\cO_{\wtd{X}},T_{\wtd{X}/Y})$ is a dg-Lie algebroid. 
In this subsection we recall definitions of three dg-algebras naturally associated with a dg-Lie algebroid: the Chevalley-Eilenberg algebra, the universal enveloping algebra, and the jet algebra. 

\subsubsection{Chevalley-Eilenberg algebra of a Lie algebroid}~\label{subsec:ce}

Let $(\mathcal A,\mathfrak g)$ be a dg-Lie algebroid with anchor map $\rho$ and Lie bracket $\mu$. 
We first consider the dg-commutative algebra $\widehat{\bS}_{\mathcal A}(\mathfrak g^\chk[-1])$, which is the adic-completion of the symmetric algebra w.r.t.~the kernel of the augmentation $\bS_{\mathcal A}(\mathfrak g^\chk[-1])\ra\mathcal A$. 
We then replace the differential $d$ by $D:=d+d_{\CE}$, where $d_{\CE}$ is defined on generators $a\in\mathcal A$ and $\xi\in\mathfrak g^\chk[-1]$ as follows\footnote{We omit (de)suspension maps from the second formul\ae.}: 
\begin{itemize}
\item $d_{\CE}(a):=s\circ \rho^\chk (da)$, and
\item the operator $d_{\CE}(\xi)\in \bS^2_\cA g^\chk[-1]$ acts on a pair of $(v,w)\in g^2$ by formula
\[ d_{\CE}(\xi)(v,w):= \rho(v)(\xi(w))-\rho(w)(\xi(v))-\xi(\mu(v,w)).\]
\end{itemize}
Then the Chevalley-Eilenberg algebra $C^*(\mathfrak g)$ is the dg-commutative algebra 
$\big(\widehat{\bS}_{\mathcal A}(\mathfrak g^\chk[-1])^\sharp,D\big)$. 
We also denote by $C^{(k)}(\mathfrak g)$ its quotient by the $k+1$-th power of the augmentation kernel. 

\begin{exa}[Relative De Rham complex]\label{exa:DR}
Let $\mathcal B\rightarrow\mathcal A$ be a morphism of (sheaves of) dg-commutative algebras, and set $\mathfrak g:=\Der_{\mathcal B}(\mathcal A)$. We equip $(\mathcal A,\mathfrak g)$ with a dg-Lie algebroid structure: $\mathfrak g$ has $[d_{\mathcal A},-]$ as differential, the graded commutator as Lie bracket, and the inclusion $\Der_{\mathcal B}(\mathcal A)\hookrightarrow \Der_{\bf k}(\mathcal A)$ as anchor map. 
The Chevalley-Eilenberg algebra is then the completion of the relative De Rham algebra $\Omega^*_{\mathcal A/\mathcal B}$. 
As a graded algebra $\Omega^*_{\mathcal A/\mathcal B}$ is $\bS_{\mathcal A^\sharp}(\Omega^1_{\mathcal A^\sharp/\mathcal B^\sharp}[-1])$, 
and the two commuting differentials are $d=L_{d_B}$ and $d_{\CE}=d_{DR}$. 
\end{exa}

\subsubsection{Universal enveloping algebra of a Lie algebroid}\label{sec:2.2.2}

We borrow the notation from the previous paragraph. 
First observe that $\mathcal A$ being acted on by $\mathfrak g$ one can consider the semi-direct product dg-Lie algebra $\mathcal A\rtimes\mathfrak g$ and take its (ordinary) universal enveloping algebra $U(\mathcal A\rtimes\mathfrak g)$, which is ${\bf k}$-augmented. 
We define the universal enveloping algebra $\U(\mathfrak g)$ of $(\mathcal A,\mathfrak g)$ as the quotient of the augmentation ideal $U^+(\mathcal A\rtimes\mathfrak g)$ by the following relations: for any $a\in\mathcal A$ and $x\in\mathcal A\oplus\mathfrak g$, $ax=a\cdot x$, where $ax$ stands for the product in the algebra $U(\mathcal A\rtimes\mathfrak g)$, and $\cdot$ denotes the $\mathcal A$-module structure on $\mathcal A\oplus\mathfrak g$. 	
It is a (non-central) dg-$\mathcal A$-algebra. 
As such it is naturally endowed with a compatible $\mathcal A$-bimodule structure. 

Moreover, $\U(\mathfrak g)$ is also endowed with the structure of a cocommutative coring in the category of {\it left} $\mathcal A$-modules: 
we have a coproduct $\Delta:\U(\mathfrak g)\to \U(\mathfrak g)\otimes_{\mathcal A}\U(\mathfrak g)$, where we only consider the {\it left} $\mathcal A$-module structure on $\U(\mathfrak g)$ when taking the tensor product. 
Notice that $\U(\mathfrak g)$ also acts on $\mathcal A$. 
All these algebraic structures and their compatibilities can be summarized in the following way: 
$\U(\mathfrak g)$ is a (dg-)bialgebroid (see \cite{BM} and references therein for a survey of the equivalent definitions). 
\begin{exa}[Relative differential operators]\label{exa:diff}
Let $\mathcal B\to\mathcal A$ and $\mathfrak g$ be as in Example \ref{exa:DR}. 
In this case $\U(\mathfrak g)$ coincides with ${\it Diff}_{\mathcal B}(\mathcal A)$, the sheaf of $\mathcal B$-linear differential operators on $\mathcal A$ endowed with the differential $[d_{\mathcal B},-]$. We easily see that the $\mathcal A$-bimodule structure factors through an action of $\mathcal A\otimes_{\mathcal B}\mathcal A$. Finally, the coproduct is given by $\Delta(P)(a,a')=P(aa')$. 
\end{exa}

\subsubsection{The jet algebra of a Lie algebroid}\label{sec:2.2.3}

We still borrow the notation from Paragraph~\ref{subsec:ce}. 
We define the jet $\mathcal A$-bimodule ${\bf J}(\mathfrak g)$ as the left $\mathcal A$-dual of $\U(\mathfrak g)$: 
${\bf J}(\mathfrak g):=\sHom_{\mathcal A\textrm{-}mod}\big(\U(\mathfrak g),\mathcal A\big)$. 
Therefore ${\bf J}(\mathfrak g)$ becomes a commutative ring in the category of {\it left} $\mathcal A$-modules. 
The $\mathcal A$-bimodule structure on ${\bf J}(\mathfrak g)$ can then be described by two dg-algebra morphisms $\mathcal A\longrightarrow{\bf J}(\mathfrak g)$: $a\longmapsto\big(P\mapsto aP(1)\big)$ and $a\longmapsto\big(P\mapsto P(a)\big)$.  

Observe that $\U(\mathfrak{g})$ is endowed with an increasing filtration obtained by assiging degree $0$, resp.~degree $1$, to elements of $\mathcal A$, resp.~$\mathfrak g$. All the algebraic structures we have seen on $\U(\mathfrak{g})$ are compatible with the filtration; in particular 
$\U(\mathfrak g)^{\leq k}$ is a sub-$\mathcal A$-coring in $\U(\mathfrak g)$. Therefore the quotient 
${\bf J}^{(k)}(\mathfrak g):= \sHom_{\mathcal A\textrm{-}mod}\big(\U(\mathfrak g)^{\leq k},\mathcal A\big)$ of ${\bf J}(\mathfrak g)$ 
inherits from it an $\mathcal A$-algebra structure. 
\begin{exa}[Relative jets]\label{exa:jets}
Let $\mathcal B\to\mathcal A$ and $\mathfrak{g}$ be as in Example \ref{exa:DR}, and denote by $\mathcal J$ the (dg-)ideal of the 
multiplication map $\mathcal A\otimes_{\mathcal B} \mathcal A\ra \mathcal A$. We have a morphism of dg-algebras
$$
\mathcal A\otimes_{\mathcal B}\mathcal A\longrightarrow {\bf J}(\mathfrak g)\,,\,a\otimes a'\longmapsto\big(P\mapsto aP(a')\big)\,.
$$
This induces morphisms $\mathcal A\otimes_{\mathcal B}\mathcal A/\mathcal J^{k+1}\longrightarrow {\bf J}^{(k)}(\mathfrak g)$, which happen to be isomorphisms whenever the underlying morphism of graded algebra $\mathcal B^\sharp\to\mathcal A^\sharp$ is smooth. 
This gives an isomorphism between $\mathcal A\hat\otimes_{\mathcal B}\mathcal A:=\underleftarrow{\lim}\,\mathcal A\otimes_{\mathcal B}\mathcal A/\mathcal J^{k+1}$ and ${\bf J}(\mathfrak g)$. 
This is essentially (a dg-version of) Grothendieck's description of differential operators via formal neighborhood of the diagonal map \cite[(16.8.4)]{EGA}. 
\end{exa}

\subsection{The Chevalley-Eilenberg algebra is the formal neighborhood of $X$ into $Y$}\label{subsec-2.3}
Let $\mathcal I:=\ker(\mathcal O_Y\to i_*\mathcal O_X)$ be the ideal sheaf of $X$ into $Y$, $\cO_X^{(k)}:=\cO_Y/\mathcal I^{k+1}$, and $\cO_X^{(\infty)}:=\underleftarrow{\lim} \,\cO_X^{(k)}$. 
We also write $\Omega^*_{\wtd{X}/Y}:=\Omega^*_{\cO_{\wtd{X}}/\cO_Y}$, and $\Omega^{(k)}_{\wtd{X}/Y}$ for its quotient by the 
$(k+1)$-th power of the augmentation ideal. 
Recall from Example \ref{exa:DR} that we have $C^*(T_{\wtd{X}/Y})=\underleftarrow{\lim} \,\Omega^{(k)}_{\wtd{X}/Y}$.
\begin{thm}\label{thm:ce}
There are quasi-isomorphisms of sheaves of dg-algebras on $Y$
$$
\phi^{(k)} : \Omega^{(k)}_{\wtd{X}/Y} \longrightarrow \cO_X^{(k)}\,.
$$
Moreover these maps are compatible with the inverse systems on both sides, and taking inverse limits induces a quasi-isomorphism of algebras
$$
\phi : C^*(T_{\wtd{X}/Y}) \longrightarrow \cO_X^{(\infty)}\,.
$$
\end{thm}

\begin{proof}
We define an $\cO_{\wtd{X}}^\sharp$-linear map $\phi:\Omega^1_{\wtd{X}_\sharp/Y}[-1] \longrightarrow \cO_Y$ by the composition
$$
\Omega^1_{\wtd{X}_\sharp/Y}[-1]\overset{-\iota_Q}{\longrightarrow}\cO_{\wtd{X}}^\sharp
= \bS_{\cO_Y}\big(E[1]\big)\overset{\epsilon}{\twoheadrightarrow}\cO_Y\,,
$$
where the last map is just the canonical $\cO_Y$-augmentation. 

\begin{lem}\label{lem:ideal}
The map $\phi$ constructed above has image inside the ideal sheaf $\cI$. 
\end{lem}
\begin{proof}
Locally on $Y$ the space of one-forms is generated by elements of the form $a d_{DR}(e)$ for local sections $a\in\cO_{\wtd{X}}$ and $e\in E$. 
Then we have
$$
\phi\big(ad_{DR}(e)\big)= \epsilon\big(-\iota_Q \big(a d_{DR}(e)\big)\big)=- \epsilon\big(a Q(e)\big) = -\epsilon(a)\epsilon\big(Q(e)\big)\,.
$$
The composed map $\cO_{\wtd{X}}\overset{\epsilon}{\to}\cO_Y\to i_*\cO_X$ being a quasi-isomorphism, we have in particular that the image of $Q(e)$ through it is zero. Therefore $\epsilon\big(Q(e)\big)\in\cI$. 
\end{proof}

By the universal property of symmetric algebras, the map $\phi$ induces a morphism of $\cO_{\wtd{X}}^\sharp$-algebras 
$$
\Omega^*_{\wtd{X}_\sharp/Y}=\bS_{\cO_{\wtd{X}}^\sharp}\Big(\Omega^1_{\wtd{X}_\sharp/Y}[-1]\Big)\longrightarrow\cO_Y\,.
$$
We shall still denote this map by $\phi$. 
\begin{lem}
$\phi$ is a morphism of $\cO_{\wtd{X}}$-algebras. In other words, it is a cochain map. 
\end{lem}

\begin{proof}
Locally on $Y$, a $k$-form $\alpha$ can be written as $fd_{DR}(e_1)\cdots d_{DR}(e_k)$ for local sections $f\in\cO_{\wtd{X}}$ and $e_i\in E$, $1\leq i\leq k$. 
We want to show that $\phi\big(D(\alpha)\big)=0$. This is a direct computation, we have
\begin{align*}
\phi&\big(D(\alpha)\big)= \phi \Big(L_Q\big(fd_{DR}(e_1)\cdots d_{DR}(e_k)\big)+d_{DR}\big(fd_{DR}(e_1)\cdots d_{DR} (e_i)\big)\Big)\\
&=\phi\Big(Q(f)d_{DR}(e_1)\cdots d_{DR}(e_k) 
+\sum_{1\leq i\leq k} (-1)^{|f|+|e_1|+\cdots +|e_{i-1}|+(i-1)}fd_{DR}(e_1)\cdots L_Q\big(d_{DR}(e_i)\big)\cdots d_{DR}(e_k)\\
&+ d_{DR}(f)d_{DR}(e_1)\cdots d_{DR}(e_k)\Big)\\
&=\epsilon\Big((-1)^k Q(f)Q(e_1)\cdots Q(e_k)+(-1)^{k+|f|+|e_1|+\cdots +|e_{i-1}|+(i-1)} \sum_{1\leq i\leq k} fQ(e_1)\cdots\big(\iota_QL_Q d_{DR}(e_i)\big) \cdots Q(e_k)\\
&+(-1)^{k+1}Q(f)Q(e_1)\cdots Q(e_k)\Big).
\end{align*}
The first and the last terms in the above sum cancel each other. 
For the middle term observe that $\iota_QL_Qd_{DR} = \iota_Q d_{DR} \iota_Q d_{DR} = QQ=0$.
\end{proof}

By Lemma~\ref{lem:ideal} the map $\phi$ sends the augmentation ideal (which is generated by $\Omega^1_{\wtd{X}/Y}[-1]$) into the ideal $\cI$. Hence $\phi$ induces an inverse system of $\cO_{\wtd{X}}$-algebra morphisms 
$$
\phi^{(k)} : \Omega_{\wtd{X}/Y}^{(k)} \longrightarrow \cO_X^{(k)}\,.
$$
We have to show that the maps $\phi^{(k)}$ defined as above are quasi-isomorphisms. 
We proceed by induction. 
For the case $k=0$ this reduces to the fact that $\cO_{\wtd{X}}$ is a resolution of $i_*\cO_X= \cO_Y/\cI$.
Assuming $\phi^{(k)}$ is a quasi-isomorphism we would like to show it is also the case for $\phi^{(k+1)}$. 
For this we consider the following morphism of short exact sequences
\[\begin{CD}
0 @>>> \bS_{\cO_Y}(E[1])\otimes_{\cO_Y}\bS^{k+1}_{\cO_Y}(E) @>>> \Omega^{(k+1)}_{\wtd{X}/Y} @>>> \Omega^{(k)}_{\wtd{X}/Y} @>>> 0 \\
@. @VVV        @V\phi^{(k+1)} VV   @V\phi^{(k)} VV  @.\\
0 @>>> i_*\Big(\bS^{k+1}_{\cO_Y}(N^\chk)\Big) @>>> \cO_X^{(k+1)} @>>> \cO_X^{(k)} @>>> 0.
\end{CD}\]
The leftmost vertical map is a quasi-isomorphism by the projection formula, and $\phi^{(k)}$ is so by the induction hypothesis. 
Thus $\phi^{(k+1)}$ is also a quasi-isomorphism. 
The Theorem is proved.
\end{proof}

\begin{rem}
This result means that the dg-Lie algebroid structure on $T_{\wtd{X}/Y}$ encodes all the formal neighborhood of $X$ into $Y$. 
Moreover our proof of the theorem works for a general closed embedding of quasi-projective dg-schemes. 
For a general morphism $X\to Y$ of quasi-projective dg-schemes this gives a \emph{definition} of the ``derived" formal neighborhood of $X$ in $Y$: take the Chevalley-Eilenberg algebra of $T_{\wtd{X}/Y}$ for a smooth factorization of $X\ra Y$ through $\wtd{X}$.
\end{rem}

\begin{exa}[The case of a global complete intersection]
As an example we now write down an explicit free resolution for the formal neighborhood algebra of an embedding $i:X\hookrightarrow Y$ in the case when $X$ is a global complete intersection in $Y$. 
It turns out that our construction recovers the classical Eagon-Northcott resolution of formal neighborhood of $X$ in $Y$ (see~\cite{BE}).

To fix the notation we assume that $X$ is defined as the zero locus of a section $s$ of a vector bundle $E^\chk$ on $Y$. 
Then the sheaf $i_*\cO_X$ has a locally free resolution of the form
$$
(\bS_{\cO_Y}(E[1]), \lrcorner s)=\cdots\overset{\lrcorner s}{\longrightarrow} \wedge^2 E \overset{\lrcorner s}{\longrightarrow} E 
\overset{\lrcorner s}{\longrightarrow} \cO_Y\longrightarrow 0\cdots\,.
$$
Note that this is a resolution of $i_*\cO_X$ as an $\cO_Y$-algebra. 
Thus we take $\wtd{X}$ to be the dg-scheme defined by $(\bS_{\cO_Y}(E[1]),\lrcorner s)$ on $Y$. 
Let us write down the Chevalley-Eilenberg complex of $T_{\wtd{X}/Y}$. 
This is by definition the completed De Rham complex of $(\bS_{\cO_Y}(E[1]),\lrcorner s)$ which is given by
$\bS_{\cO_Y}(E[1])\otimes_{\cO_Y} \widehat{\bS}_{\cO_Y}(E)$ as a quasi-coherent sheaf on $Y$, with differential acting on the component $\bS_{\cO_Y}^i(E[1])\otimes_{\cO_Y}\bS_{\cO_Y}^j(E)$ by the sum of the following two operators:
\begin{align*}
L_{\lrcorner s}\big((e_1\cdots e_i)\otimes (f_1\cdots f_j)\big) &
= \sum_{l=1}^i (-1)^{l-1}\Big(\big(e_1\cdots (\langle s,e_l\rangle) \cdots e_i\big)\otimes (f_1\cdots f_j)\Big);\\
d_{DR} \big(( e_1\cdots e_i)\otimes (f_1\cdots f_j)\big) &
= \sum_{l=1}^i (-1)^{l-1}\big((e_1\cdots \widehat{e_l} \cdots e_i)\otimes (e_lf_1\cdots f_j)\big).
\end{align*}
\end{exa}

\subsection{The jet algebra represents $i^*i_*$}\label{sec-2.4}

Our goal in this subsection is to understand the jet algebra ${\bf J}(T_{\wtd{X}/Y})$ of the dg-Lie algebroid $T_{\wtd{X}/Y}$. 
Our main result identifies ${\bf J}(T_{\wtd{X}/Y})$ with the formal neighborhood of $\wtd{X}$ inside $\wtd{X}\times_Y \wtd{X}$. 

\medskip

Let us begin with a description of the derived self-intersection $\wtd{X}\times_Y\wtd{X}$: 
by the definition of fiber product of dg-schemes its underlying scheme is $\wtd{X}^0\times_Y \wtd{X}^0=Y$ (recall that $\wtd{X}^0=Y$) and its structure sheaf is $\cO_{\wtd{X}}\otimes_{\cO_Y} \cO_{\wtd{X}}$. 
\begin{comment}
Then recall from Example \ref{exa:diff} that $\U(T_{\wtd{X}/Y})$ is the algebra ${\it Diff}_{\cO_Y}(\cO_{\wtd{X}})$ of $\cO_Y$-linear differential operators on $\cO_{\wtd{X}}$, and that it is endowed with a natural $\cO_{\wtd{X}}\otimes_{\cO_Y} \cO_{\wtd{X}}$-module structure. 
As such it lives over the derived self-intersection $\wtd{X}\times_Y\wtd{X}$. 
\end{comment}
The diagonal map $\Delta_{\wtd{X}/Y}: \wtd{X} \ra \wtd{X}\times_Y \wtd{X}$ is defined by the identity map on the underlying scheme $Y$ and the multiplication map $\cO_{\wtd{X}}\otimes_{\cO_Y}\cO_{\wtd{X}}\ra \cO_{\wtd{X}}$ on the structure sheaves. 
We denote its kernel by $\mathcal J$, and we define $\cO^{(k)}_{\Delta}:=\cO_{\wtd{X}\times_Y\wtd{X}}/\mathcal J^{k+1}$ ($k\geq0$) and $\cO^{(\infty)}_{\Delta}:=\underleftarrow{\lim} \,\cO_\Delta^{(k)}$. 
According to Example \ref{exa:jets} we have the following description of ${\bf J}(T_{\wtd{X}/Y})$. 
\begin{prop}\label{thm:ug}
There is an isomorphism $\cO^{(\infty)}_{\Delta} \cong {\bf J}(T_{\wtd{X}/Y})$ of $\mathcal O_{\wtd{X}\times_Y\wtd{X}}$-algebras. 
\end{prop}

This result in fact holds for any morphism of dg-schemes $X\to Y$. 
The specificity of a closed embedding $i:X\hookrightarrow Y$ of varieties is that the formal neighborhood of the diagonal in $\wtd{X}\times_Y\wtd{X}$ is quasi-equivalent to $\wtd{X}\times_Y\wtd{X}$ itself, as we prove now. 

\begin{thm}\label{thm:complete}
For a closed embedding $i:X\hookrightarrow Y$ of smooth schemes assume that $Y$ is quasi-projective. 
Then the formal neighborhood algebra $\cO^{(\infty)}_\Delta$ of the embedding $\wtd{X}\hookrightarrow \wtd{X}\times_Y \wtd{X}$ is quasi-isomorphic to the structure sheaf $\cO_{\wtd{X}\times_Y\wtd{X}}$.
\end{thm}

\noindent In paticular, combined with Proposition \ref{thm:ug}, we conclude that ${\bf J}(T_{\wtd{X}/Y})$ is the structure sheaf of $\wtd{X}\times_Y \wtd{X}$. 

\begin{proof}
There is a natural map of sheaves of algebras $\cO_{\wtd{X}\times_Y \wtd{X}}\ra \cO^{(\infty)}_\Delta$ induced from the natural quotient maps $\cO_{\wtd{X}\times_Y\wtd{X}}\ra \cO^{(k)}_\Delta$. Moreover as the derived tensor product the cohomologies of $\cO_{\wtd{X}\times_Y\wtd{X}}$ are given by $\underline{\Tor}^i_Y(\cO_X,\cO_X)$ where the cohomological grading is compatible with the filtration degree on $\cO_{\wtd{X}\times_Y\wtd{X}}$. Since $\underline{\Tor}^i_Y(\cO_X,\cO_X)$ vanishes for large $i$ for a smooth space $Y$ and the dg-ideal $\mathcal J$ of the diagonal map $\cO_{\wtd{X}\times_Y\wtd{X}}\ra\cO_{\wtd{X}}$ is concentrated on strictly negative degrees, the natural quotient maps $\cO_{\wtd{X}\times_Y\wtd{X}}\ra \cO^{(k)}_\Delta$ are quasi-isomorphisms for large $k$. The theorem follows after taking inverse limit. (Note that taking inverse limit does not in general commute with taking cohomology. However in our case this is true since the cohomologies of the inverse system stabilize.)
\end{proof}
\begin{proof}[Proof of Theorem \ref{thm:intro2}]
We should prove that, for a closed embedding $i:X\hookrightarrow Y$, ${\bf J}(T_{\wtd{X}/Y})$ represents the functor $i^*i_*$. 
This has to be interpreted correctly. Namely we have factored the map $i$ as $\pi \circ j$ for a quasi-equivalence $j: X\rightarrow \wtd{X}$ 
and a smooth map $\pi:\wtd{X}\ra Y$. Since the map $\pi$ is smooth, the sheaf $h_*\cO_{\wtd{X}\times_Y \wtd{X}}$ where $h$ is the embedding 
$\wtd{X}\times_Y \wtd{X}\hookrightarrow \wtd{X}\times\wtd{X}$, gives the kernel for $\pi^*\pi_*: \D(\wtd{X})\ra \D(\wtd{X})$. 
Thanks to Proposition \ref{thm:ug} and Theorem \ref{thm:complete} we have a quasi-isomorphism
$$
\cO_{\wtd{X}\times_Y \wtd{X}} \cong {\bf J}(T_{\wtd{X}/Y})\,.
$$ 
Thus we conclude that $h_*{\bf J}(T_{\wtd{X}/Y})\in \D(\wtd{X}\times\wtd{X})$ is the kernel representing the functor $\pi^*\pi_*$. Finally we observe that the functor $\pi^*\pi_*$ corresponds to $i^*i_*$ via the equivalence $\D(X)\cong \D(\wtd{X})$. Hence the sheaf $(j\times j)^*h_*{\bf J}(T_{\wtd{X}/Y})$ on $X\times X$ represents $i^*i_*$. 

As for the monad structure on $\pi^*\pi_*$, it {\bf is} the algebra structure on $\cO_{\wtd{X}\times_Y \wtd{X}}$. 
\end{proof}
\begin{rem}
It is interesting to observe that there is an iterative process arising from an embedding of algebraic varieties $i:X\hookrightarrow Y$. 
Namely we have associated to the embedding a natural dg-Lie algebroid $T_i$ (here we changed the notation for $T_{\wtd{X}/Y}$ to emphasize the morphism $i$). 
But the morphism $i$ defines another embedding of dg-schemes $\Delta_i:\wtd{X} \ra \wtd{X}\times_Y \wtd{X}$. The dg-Lie algebroid associated to this latter embedding $\Delta_i$ will be denoted by $T_{\Delta_i}$. There is an interesting relationship between $T_i$ and $T_{\Delta_i}$ in view of Theorem~\ref{thm:ce} and Proposition~\ref{thm:ug}:
\[ C^*(T_{\Delta_i})\stackrel{\mbox{\ref{thm:ce}}}{\longrightarrow} \cO_{\wtd{X}/\wtd{X}\times_Y\wtd{X}}^{(\infty)}
\stackrel{\mbox{\ref{thm:ug}}}{\longrightarrow} {\bf J}(T_i).\]
In particular we see that the Lie algebroid $T_{\Delta_i}$ encodes less information than that of $T_i$ since in the above isomorphism we do not see the Lie structure on $T_i$ from that of $T_{\Delta_i}$. Moreover observe that $T_{\Delta_i}\cong T_i[-1]$, hence its Chevalley-Eilenberg algebra is of the form $\widehat{\bS}(T_i^\chk)$ endowed with the Chevalley-Eilenberg differential. 
The isomorphism $C^*(T_{\Delta_i})\cong {\bf J}(T_i)$ then certainly resembles the $\PBW$ theorem for $T_i$. 
It is thus natural from this isomorphism to conjecture that the $\PBW$ property holds for $T_i$ if-and-only-if certain Lie structure vanishes on 
$T_{\Delta_i}$. This is in fact the main point of the paper~\cite{AC} where the authors study the $\HKR$ property for an embedding $i:X\ra Y$. 
Their result on the derived level asserts that the vanishing of the Lie structure on $N[-2]=T_{\Delta_i}$ is equivalent to the $\HKR$ property 
for $N[-1]=T_i$.
\end{rem}

\subsection{The universal enveloping algebra represents $i^*i_\dagger $}\label{sec-2.5}

First we recall that the functor $i_\dagger$ is the left adjoint of 
\[ i^*: {\bf D}^b_{coh}(Y) \ra {\bf D}^b_{coh}(X).\]
Explicitly for an object $E\in {\bf D}^b_{coh}(X)$, we have
\[ i_\dagger(E)= i_*(E\otimes \det N) [\dim X- \dim Y].\]
Let us sketch a proof of Theorem \ref{thm:intro}. 
By the universal property of $\U(T_{\widetilde{X}/Y})$ we get a morphism 
$\U(T_{\widetilde{X}/Y})\longrightarrow\mathcal{H}\mathrm{om}_{\mathcal O_Y}(\mathcal O_{\widetilde{X}},\mathcal O_{\widetilde{X}})$ 
of dg-algebras in $\mathcal O_{\widetilde{X}}$-bimodules over $\mathcal O_Y$. One can prove that, in the case of a closed embedding, 
it is a quasi-isomorphism (the proof is very similar to the one of Theorem \ref{thm:complete}, with $\underline{\Tor}$'s being replaced 
by $\underline{{\rm Ext}}$'s). In particular we have an algebra isomorphism 
$\U(i_*N[-1])\to\mathcal E\mathrm{xt}_{\mathcal O_Y}(i_*\mathcal O_X,i_*\mathcal O_X)$ in ${\bf D}^b_{coh}(Y)$, and thus 
by seeing $\U(i_*N[-1])$ as an object of ${\bf D}^b_{coh}(X\times X)$ set-theoretically supported on the diagonal, 
we get that it is the kernel representing the functor $i^*i_\dagger :{\bf D}^b_{coh}(X)\to {\bf D}^b_{coh}(X)$. 
The monad structure on $i^*i_\dagger $ coming from the product on $\mathcal E\mathrm{xt}_{\mathcal O_Y}(i_*\mathcal O_X,i_*\mathcal O_X)$ 
easily identifies with the one coming from the projection formula 
$i^*i_\dagger i^*i_\dagger \Rightarrow i^*i_\dagger $. \hfill\qed

\medskip

Let us provide yet another approach. Borrowing the notation from Paragraph~\ref{sec:2.2.2}, we have the following
\begin{lem}
The functor 
$\U(\mathfrak g)\underset{\mathcal A}{\otimes}-:\mathcal A\textrm{-}mod\longrightarrow\mathcal A\textrm{-}mod$ is left adjoint 
to $\mathsf{Hom}_{\mathcal A\textrm{-}mod}\big(\U(\mathfrak g),-)$. Here we use that for any left $\mathcal A$-module $M$, the right 
$\mathcal A$-module structure on $\U(\mathfrak g)$ turns $\mathsf{Hom}_{\mathcal A\textrm{-}mod}\big(\U(\mathfrak g),M\big)$ 
into a left $\mathcal A$-module. 
\end{lem}
\begin{proof}
Let $M,N$ be $\mathcal A$-modules. As usual the one-to-one correspondence
$$
\phi\in\mathsf{Hom}_{\mathcal A\textrm{-}mod}\left(\U(\mathfrak g)\underset{\mathcal A}{\otimes}M,N\right)\tilde\longleftrightarrow
\,\mathsf{Hom}_{\mathcal A\textrm{-}mod}\Big(M,\mathsf{Hom}_{\mathcal A\textrm{-}mod}\big(\U(\mathfrak g),N\big)\Big)\ni\psi
$$
is given by $\phi(P\otimes m)=\psi(m)(P)$. 
\begin{comment}It is obviously well-defined
\begin{itemize}
\item $\phi(Pa\otimes m-P\otimes am)=\psi(m)(Pa)-\psi(am)(P)=(a\psi)(m)(P)-\psi(am)(P)=0$. 
\item $\psi(m)(aP)-a\big(\psi(m)(P)\big)=\phi(aP\otimes m)-a\phi(P\otimes m)=0$. 
\end{itemize}\end{comment}
\end{proof}
Then notice that $\mathsf{Hom}_{\mathcal A\textrm{-}mod}\big(\U(\mathfrak g),-)={\bf J}(\mathfrak{g})\underset{\mathcal A}{\widehat\otimes}-$, 
where the $\mathcal A$-bimodule structure on ${\bf J}(\mathfrak{g})$ is the one described in Paragraph~\ref{sec:2.2.3}. 

\medskip

We finally consider the case when $\mathcal A=\mathcal O_{\widetilde{X}}$ and $\mathfrak{g}=T_{\wtd{X}/Y}$. 
It follows from Theorem \ref{thm:intro2} (which we proved in the previous Paragraph) that $\U(T_{\wtd{X}/Y})$ represents 
the kernel for the left adjoint functor to $i^*i_*$, which is $i^*i_\dagger $. 
This has to be understood as $(j\times j)^*h_*\U(T_{\wtd{X}/Y})$, which is nothing but $\U(i_*N[-1])$ viewed as an object of 
${\bf D}^b_{coh}(X\times X)$, being the kernel representing $i^*i_\dagger $. 

%\newpage
\section{Monads}\label{sec:mon}

In the previous Section we have identified $\U(T_{\widetilde{X}/Y})$ with the kernel of $i^*i_\dagger $. 
It happens that both are equipped with an associative product (they even both carry a Hopf-like structure). 
Identifying these additional structures is the subject of the present Section. 

\subsection{The Hopf monad associated with the universal enveloping algebra}

Let $(\mathcal A,\mathfrak g)$ be a Lie algebroid. 
We have seen that $\U(\mathfrak g)$ is a bialgebroid. 
It actually has a very specific feature: source and target maps $\mathcal A\to\U(\mathfrak g)$ are the same. 
Therefore, the forgetful functor ${\it U}:\U(\mathfrak g)\textrm{-}mod\longrightarrow \mathcal A\textrm{-}mod$ is {\it strong} monoidal 
(recall that $\U(\mathfrak g)$ being a bialgebroid its category of left modules is monoidal, see e.g.~\cite{BM} and references 
therein)\footnote{The forgetful functor usually goes to $\mathcal A$-bimodules, but here its essential image is the monoidal subcategory consisting 
of those bimodules which have the same underlying left and right module structure. It is isomorphic to the monoidal category of 
$\mathcal A$-modules. }. 

Observe that ${\it U}$ has a left adjoint: 
${\it F}:\U(\mathfrak g)\underset{\mathcal A}{\otimes}-:\mathcal A\textrm{-}mod\longrightarrow\U(\mathfrak g)\textrm{-}mod$. 
Moreover, ${\it U}$ being strong monoidal, then its left adjoint ${\it F}$ is {\it colax} monoidal and hence 
the monad ${\it T}:={\it U}{\it F}$ is a {\it Hopf monad} in the sense of \cite{Mo}: it is a monad in the $2$-category 
$OpMon$ having monoidal categories as objects, colax monoidal functors as $1$-morphisms and natural transformations of those as $2$-morphisms. 

\subsubsection{The dual Hopf comonad associated with the jet algebra}

Notice that the strong monoidal functor ${\it U}$ also has a right adjoint ${\it G}:=\sHom_{\mathcal A\textrm{-}mod}\big(\U(\mathfrak g),-)$, 
which is {\it lax} monoidal. Going along the same lines as above one sees that ${\it S}:={\it U}{\it G}$ (which is right adjoint to ${\it T}$) 
is a {\it Hopf comonad}, meaning that it is a comonad in the $2$-category $Mon$ having monoidal categories as objects, lax monoidal functors 
as $1$-morphisms and natural transformations of those as $2$-morphisms. 

Finally recall that $\sHom_{\mathcal A\textrm{-}mod}\big(\U(\mathfrak g),-)\cong{\bf J}(\mathfrak{g})\underset{\mathcal A}{\widehat\otimes}-$, 
where the $\mathcal A$-bimodule structure on ${\bf J}(\mathfrak{g})$ is the one described in  Paragraph~\ref{sec:2.2.3}. 
Notice that, on the one hand, the lax monoidal structure on ${\it S}$ is given by the coproduct on $\U(\mathfrak g)$, and thus by the product 
on ${\bf J}(\mathfrak{g})$: 
$$
\Big({\bf J}(\mathfrak{g})\underset{\mathcal A}{\widehat\otimes}-\Big)
\otimes_{\mathcal A}\Big({\bf J}(\mathfrak{g})\underset{\mathcal A}{\widehat\otimes}-\Big)
\cong\big({\bf J}(\mathfrak{g})\otimes_{\mathcal A}{\bf J}(\mathfrak{g})\big)\underset{\mathcal A\otimes \mathcal A}{\widehat\otimes}(-\otimes-)
\Longrightarrow{\bf J}(\mathfrak{g})\underset{\mathcal A}{\widehat\otimes}(-\otimes_{\mathcal A}-)\,.
$$
On the other hand, the comonad structure on ${\it S}$ is given by the product on $\U(\mathfrak g)$, and hence by the coproduct 
on ${\bf J}(\mathfrak g)$: 
$$
{\bf J}(\mathfrak{g})\underset{\mathcal A}{\widehat\otimes}-\Longrightarrow
\Big({\bf J}(\mathfrak{g})\underset{\mathcal A}{\widehat\otimes}{\bf J}(\mathfrak{g})\Big)\underset{\mathcal A}{\widehat\otimes}-\cong
{\bf J}(\mathfrak{g})\underset{\mathcal A}{\widehat\otimes}\Big({\bf J}(\mathfrak{g})\underset{\mathcal A}{\widehat\otimes}-\Big)\,.
$$

\subsection{The Hopf (co)monad associated with a closed embedding}

Similarly to the above, the left adjoint $i_\dagger $, resp.~the right adjoint $i_*$, to the strong monoidal functor $i^*$ is colax monoidal, 
resp.~lax monoidal. Hence $i^*i_\dagger $ is a Hopf monad and $i^*i_*$ is a Hopf comonad. 

We already know (see Subsections~\ref{sec-2.4} and~\ref{sec-2.5}) that there are isomorphisms of functor $i^*i_*\cong {\it U}{\it G}$ and 
$i^*i_\dagger \cong {\it U}{\it F}$ whenever $(\mathcal A,\mathfrak g)=(\mathcal O_{\wtd{X}},{\bf T}_{\wtd{X}/Y})$.  
It remains to be shown that the Hopf (co)monad structures coincide. 

\subsubsection{The Hopf comonad associated with a morphism of algebras}

In this paragraph, functors are not derived. Let $\mathcal B\to\mathcal A$ be a morphism of (sheaves of dg-)algebras. 
The strong monoidal (dg-)functor $\mathcal A\otimes_{\mathcal B}-:\mathcal B\textrm{-}mod\to\mathcal A\textrm{-}mod$ 
admits the ``forgetful'' functor $\mathcal A\textrm{-}mod\to\mathcal B\textrm{-}mod$ as a right adjoint, which is then lax monoidal. 
Again, this turns $\mathcal A\otimes_{\mathcal B}-:\mathcal A\textrm{-}mod\to\mathcal A\textrm{-}mod$ into a (dg-)Hopf comonad. 
Notice that the lax monoidal structure 
$(\mathcal A\otimes_{\mathcal B}-)\otimes_{\mathcal A}(\mathcal A\otimes_{\mathcal B}-)\Rightarrow\mathcal A\otimes_{\mathcal B}(-\otimes_{\mathcal A}-)$
is given by the product of $\mathcal A$ while the comonad structure 
$\mathcal A\otimes_{\mathcal B}-\Rightarrow\mathcal A\otimes_{\mathcal B}(\mathcal A\otimes_{\mathcal B}-)$ 
is given by the morphism $\mathcal A\to\mathcal A\otimes_{\mathcal B}\mathcal A,a\mapsto a\otimes 1$. 

Observe that this Hopf comonad can be seen as the Hopf comonad associated with a Hopf algebroid 
in a similar way to what happens with the jet algebra. Namely, the algebra $\mathcal A\otimes_{\mathcal B}\mathcal A$ 
is equipped with two obvious algebra maps $\mathcal A\to\mathcal A\otimes_{\mathcal B}\mathcal A$
and a coproduct 
\begin{eqnarray*}
\Delta:\mathcal A\otimes_{\mathcal B}\mathcal A & \longrightarrow & 
(\mathcal A\otimes_{\mathcal B}\mathcal A)\underset{\mathcal A}{\otimes}(\mathcal A\otimes_{\mathcal B}\mathcal A)
\cong \mathcal A\otimes_{\mathcal B}\mathcal A\otimes_{\mathcal B}\mathcal A \\
a\otimes a' & \longmapsto & a\otimes 1\otimes a'
\end{eqnarray*}
which satisfy (similarly to ${\bf J}(\mathfrak g)$) the axioms of a cogroupoid object in (sheaves of dg-)commutative algebras. 
This proves the functor 
$(\mathcal A\otimes_{\mathcal B}\mathcal A)\underset{\mathcal A}{\otimes}-:\mathcal A\textrm{-}mod\to\mathcal A\textrm{-}mod$
with the structure of a Hopf comonad: as in the case of ${\bf J}(\mathfrak g)$, the lax monoidal structure comes from the product 
of $\mathcal A\otimes_{\mathcal B}\mathcal A$ while the comonad structure comes from its coproduct. 
\begin{prop}\label{prop-3.1}
There is a natural isomorphism $\mathcal A\otimes_{\mathcal B}-\tilde\to(\mathcal A\otimes_{\mathcal B}\mathcal A)\underset{\mathcal A}{\otimes}-$ 
of Hopf comonads. 
\end{prop}
\begin{proof}
Let $\mathcal M$ be an $\mathcal A$-module. One observes that the natural isomorphism 
$\mathcal A\otimes_{\mathcal B}\mathcal M\tilde\to(\mathcal A\otimes_{\mathcal B}\mathcal A)\underset{\mathcal A}{\otimes}\mathcal M$ is given by 
$a\otimes m\mapsto a\otimes 1\otimes m$ and obviously commutes with the comonad and lax monoidal structures. 
\end{proof}

\subsubsection{Identifying $i^*i_*$ and ${\it U}{\it G}$ as Hopf comonads}

\begin{prop}\label{prop-3.2}
Let $\mathcal B\to\mathcal A$ and $\mathfrak g$ be as in Example \ref{exa:DR}. Then the dg-algebra morphism 
$$
\mathcal A\otimes_{\mathcal B}\mathcal A\longrightarrow {\bf J}(\mathfrak{g})
$$
introduced in Example \ref{exa:jets} is a dg-Hopf algebroid morphism. 
\end{prop}
\begin{proof}
Straightforward. 
\end{proof}
Let $\mathcal A=\cO_{\widetilde{X}}$ and $\mathfrak{g}=T_{\widetilde{X}/Y}$. 
Then it follows from Theorem \ref{thm:complete} that 
$\mathcal A\otimes_{\mathcal B}\mathcal A\to\mathcal A\hat\otimes_{\mathcal B}\mathcal A\cong{\bf J}(\mathfrak g)$ is a quasi-isomorphism. 
Recall that $\mathcal A\otimes_{\mathcal B}\mathcal A=\cO_{\widetilde{X}\times_Y\widetilde{X}}$ and 
$\mathcal A\hat\otimes_{\mathcal B}\mathcal A=\cO_{\Delta}^{(\infty)}$. The next result then follows from Propositions 
\ref{prop-3.1} and \ref{prop-3.2}. 
\begin{thm}
There is a natural quasi-isomorphism $\pi^*\pi_*\Rightarrow{\it S}$ of dg-Hopf comonads 
on $\cO_{\widetilde{X}}\textrm{-}mod$. 
\end{thm}
If one denotes by the same symbols dg-functors and the induces functors on the homotopy category, one then has 
\begin{cor}
The Hopf comonads $i^*i_*$ and ${\it S}$ on $\D(X)$ are natually isomorphic. Therefore their respective left adjoint Hopf monads 
$i^*i_\dagger $ and ${\it T}$ are isomorphic too. 
\end{cor}

\section{Deformation theory of sheaves of algebras}
\label{sec:deform}

Let ${\bf k}$ be a base field which is of characteristic zero. All algebras are $k$-algebras, and derivations of algebras are ${\bf k}$-linear. In this section we recollect basic facts on the folklore that deformations of a sheaf of commutative algebras $\cA$ is governed by a differential graded Lie algebra resolving the tangent lie algebra $T_\cA$. To ensure convergence we restrict to a pronilpotent Lie subalgebra $T^+_\cA$ of $T_\cA$, and assume that the exponential map
\[\exp: T^+_\cA \ra \Aut(\cA)\]
is well-defined. We denote the image $\exp(T^+_\cA)$ by $\Aut^+(\cA)$. This is a sheaf of prounipotent groups.

\subsection{Thom-Whitney resolutions of sheaves of algebras}

Let $\mathcal U:=\{ U_i \} _{i\in \cI}$ be an open covering of a topological space $X$. 
For an open subset $U_i\subset X$ we denote by $\rho_i$ the inclusion $U_i\hookrightarrow X$. 
For a multi-index $I=(i_0,\cdots, i_k)\in \cI^{k+1}$ we denote by $U_I$ the intersection $U_{i_0}\cap\cdots U_{i_k}$, and by $\rho_I$ the inclusion $U_I\hookrightarrow X$. Associated to this data is a simplicial space 
$$
\mathcal U_\bullet:=\left(\cdots \coprod_{I\in \cI^{k+1}} U_I
\cdots\triplearrows \coprod_{(i_0,i_1)\in \cI^2} U_{i_0,i_1} 
\rightrightarrows\coprod_{i_0\in\cI}U_{i_0} \right)\,.
$$
Let $\mathcal A$ be a sheaf of algebras on a topological space $X$. We get a cosimplicial sheaf of algebras 
$$\mathcal A_{\mathcal U}^\bullet:=\left(\coprod_{i_0\in\cI}(\rho_{i_0})_*(\rho_{i_0})^*\mathcal A \rightrightarrows \coprod_{(i_0,i_1)\in \cI^2} (\rho_{i_0,i_1})_*(\rho_{i_0,i_1})^*\mathcal A
\triplearrows \cdots\cdots \coprod_{I\in \cI^{k+1}} (\rho_I)_*(\rho_I)^*\mathcal A
\right)$$ 
on $X$. Its global sections $\mathcal A_{\mathcal U}^\bullet(X)$ form a cosimplicial algebra.

The (sheaf version) \v Cech complex ${\mathcal C}^*(\cU, \cA):=\Tot(\cA_\cU^\bullet)$ (or simply ${\mathcal C}^*(\cA)$) of the cosimplicial sheaf of algebras $\mathcal A_{\mathcal U}^\bullet$ gives a resolution of $\mathcal A$. We denote by $C^*(\cA)$ its global sections which form a differential graded algebra. However if $\cA$ were commutative, the \v Cech resolution $\mathcal C^*(\cA)$ is in general not commutative. For purposes of this paper we need to consider commutative (or Lie) resolution of sheaves of commutative (Lie respectively) algebras. Since the \v Cech resolution does not respect the symmetric monoidal structure, we need to use a better resolution.

\subsubsection{Thom-Whitney complex} 
We now describe a symmetric monoidal functor $TW$ from cosimplicial cochain complexes to cochain complexes. We refer to the Appendix in~\cite{FMM} for a more detailed discussion. Given a cosimplicial cochain complex 
$V^\bullet= V^0\rightrightarrows V^1\triplearrows V^2 \cdots$, we define a cochain complex $\TW(V^\bullet)$ as follows. Denote by
\[ d^{n,k}: V^{n-1}\ra V^n\]
the structure maps of $V^\bullet$ for $n\geq 1$, and $0\leq k\leq n$. Denote by $\Omega^*_{\Delta^n}$ algebraic differential forms on the $n$-simplex. More precisely, this is the differential graded commutative algebra generated by degree zero variables $t_0,\cdots,t_n$ and degree one variables $dt_0,\cdots,dt_n$ with relations:
\[ \Omega^*_{\Delta^n}:= k[t_0,\cdots,t_n,dt_0,\cdots,dt_n]/\langle t_0+\cdots+t_n=1, dt_0+\cdots+dt_n=0\rangle.\]
Its differential is defined by $d(t_i)=dt_i$ and $d(dt_i)=0$ on generators. The collection $\{\Omega_{\Delta^n}^*\}_{n=0}^\infty$ form a simplicial ring whose structure map
\[ \partial_{n,k}: \Omega^*_{\Delta^n} \ra \Omega^*_{\Delta^{n-1}} \]
is the pull-back map induced by the $k$-th ($0\leq k \leq n$) inclusion map $\Delta^{n-1}\hookrightarrow \Delta^n$ of standard simplices. The cochain complex $\TW(V^\bullet)$ is defined to be the sub-complex of the product
\[\prod_{n=0}^\infty  V^n \otimes \Omega^*_{\Delta^n},\]
consisting of elements of the form $\prod_{n=0}^\infty x_n$ such that
\[ (d^{n+1,k}\otimes \id) x_n = (\id\otimes \partial_{n+1, k}) x_{n+1}\;\; \forall n\geq 0, 0\leq k \leq n+1.\]
In short, the cochain complex $\TW(V^\bullet)$ is simply the equalizer
$$
\TW (V^\bullet)=
\eq\left(\prod_n V^n\otimes \Omega^*_{\Delta^n}\doublerightarrow{\varphi_*\otimes \id}{\id\otimes\varphi^*}\prod_{\varphi\in\Delta([k],[l])} V^k\otimes \Omega^*_{\Delta^l}\right)\,.
$$
The cochain complex $\TW(V^\bullet)$ is quasi-isomorphic to $\Tot(V^\bullet)$. In fact there is a homotopy retraction between the two complexes
\begin{align*}
I: \Tot(V^\bullet) &\ra \TW(V^\bullet)\\
P: \TW(V^\bullet) &\ra \Tot(V^\bullet)\\
H: \TW(V^\bullet) &\ra \TW(V^\bullet)
\end{align*}
where the maps $I$, $P$ and $H$ are given by explicit formulas.

The functor $\TW$ provides a nice (i.e. functorial) way of describing homotopy limits of cosimplicial diagrams within the category of cochain complexes. Being symmetric monoidal  it sends commutative, Lie and associative cosimplicial dg-algebras 
to commutative, Lie and associative dg-algebras, respectively. It also sends cosimplicial dg-modules (over a cosimplicial dg-algebra) 
to dg-modules, and preserves the tensor product of those (because $\TW$ commutes with finite colimits, and thus with push-outs). Observe that $\TW$ extends to  complete topological cochain complexes, with the completed tensor product $\hat\otimes$ as monoidal product.

Given a sheaf of algebras $\mathcal A$ and an open cover $\mathcal U$, we have a cosimplicial sheaf of algebras $\cA_\cU^\bullet$ as described above. Applying the Thom-Whitney functor (in the category of sheaves) we get a quasi-isomorphism of sheaves of 
differential graded algebras 
\[ \mathcal A\longrightarrow \cC^*(\cU,\cA)=\Tot(\cA_\cU^\bullet) \stackrel{I}{\longrightarrow} \TW(\mathcal A_{\mathcal U}^\bullet).\]
This provides a flasque resolution of $\mathcal A$ as soon as $\mathcal A$ is flasque on every open subset $U_{i_0}$. In the following we use the notation $\cC^*_{\TW}(\cU,\cA)$ (or simply $\cC^*_\TW(\cA)$) to denote the Thom-Whitney resolution. Its global sections will be denoted by $C^*_\TW(\cA)$. Due to the monoidal properties of the functor $\TW$, the resolution $\cC^*_\TW(\cA)$ is commutative (or Lie) whenever $\cA$ is.

\subsubsection{The Totalization functor}
We will need to descibe another kind of homotopy limit. Namely, there is a functor $\Tot$~\footnote{We used the same notation $\Tot$ when forming total complexes, but this should not cause any confusion.} from cosimplicial simplicial sets to simplicial sets defined, on a given cosimplicial simplicial set $X^\bullet$, by 
$$
\Tot(X^\bullet):=
\eq\left(\prod_n(X^n)^{\Delta^n}\doublerightarrow{\varphi_*}{\varphi^*}\prod_{\varphi\in\Delta([k],[l])}(X^k)^{\Delta^l}\right)\,.
$$
Notice that $(X^k)^{\Delta^l}=X^k_l$, where we put the simplicial degree in the lower index. 
The functor $\Tot$ provides a nice way of describing homotopy limits of cosimplicial diagrams within 
the category of simplicial sets.

\subsection{Maurer-Cartan elements versus non-abelian $1$-cocycles}\label{MCvsNA}
We shall define two groupoids naturally associated to a cosimplicial pronilpotent Lie algebra $\gog^\bullet$. Recall the Maurer-Cartan set $\MC(\gog)$ of a pronilpotent differential graded Lie algebra is defined by
\[ \MC(\gog):=\big\{ \theta\in \gog_1\mid d\theta+\frac{1}{2}[\theta,\theta]=0\big\}.\]
A degree zero element $a\in \gog_0$ acts on the set $\MC(\gog)$ by
\[ \theta \mapsto e^a*\theta:= e^{\ad(a)}(\theta)-\frac{e^{\ad(a)}-1}{\ad(a)} (da).\]
The Deligne groupoid $\Del(\mathfrak g)$ of $\gog$ is the action groupoid associated to this action. We define the Deligne groupoid of a cosimplicial pronilpotent Lie algebra $\gog^\bullet$ to be $\Del( \TW(\gog^\bullet))$.

\medskip

The second groupoid $Z^1(\exp(\gog^\bullet))$ associated to $\gog^\bullet$ is defined as follows. 
We denote by $G^\bullet:=\exp(\gog^\bullet)$ the cosimplicial prounipotent group associated to $\gog^\bullet$~\footnote{As a set $G$ is $\mathfrak g$ and the group structure is given 
by the Campbell-Hausdorff formula. }. The set of objects of $Z^1(\exp(\gog^\bullet))$ is the set of nonabelian $1$-cocycles, i.e. $T\in \gog^1$ such that $e^{\partial_0(T)} e^{-\partial_1(T)} e^{\partial_2(T)}=\id$. An element $a\in G^0$ acts on the set of nonabelian $1$-cocycles by
\[ e^T\mapsto e^{-\partial_1(a)} e^T e^{\partial_0(a)}.\]
We define $Z^1(\exp(\gog^\bullet))$ to be the groupoid associated to this action.

\begin{thm}\label{thm:mc}
Let $\mathfrak g^\bullet$ be a cosimplicial pronilpotent Lie algebra and $G^\bullet:=\exp(\mathfrak g^\bullet)$ the 
corresponding cosimplicial ${\bf k}$-prounipotent group. Then there is an equivalence
$$
\Del(\TW(\gog^\bullet)) \cong Z^1(\exp(\gog^\bullet)).
$$
\end{thm}

\begin{proof}
We start with some recollection of homotopy theory. 
We temporary forget about the cosimplicial structure and let $\mathfrak g$ be a pronilpotent Lie algebra, and $G:=\exp(\mathfrak g)$. 
We consider the two following simplicial sets: the nerve $N(G):=\left(\cdots G^2\triplearrows G\rightrightarrows *\right)$ of the group $G$ and the simplicial set $\MC_{\bullet}(\mathfrak g)$ defined by 
$$
\MC_{\bullet}(\mathfrak g):=\MC\left(\Omega_{\Delta^\bullet}\hat\otimes\mathfrak g\right)=
\left\{\theta\in\mathfrak \Omega^1_{\Delta^\bullet}\hat\otimes\mathfrak g\Big|d_{dR}(\theta)+\frac12[\theta,\theta]=0\right\}\,.
$$
%Notice that $\pi_0\circ MC_\bullet=\overline{MC}$. 
\begin{lem}[Folklore]
$N(G)$ is weakly equivalent to $\MC_{\bullet}(\mathfrak g)$. 
\end{lem}
\begin{proof}[Sketch of proof]
We construct an explicit map $e:\MC_{\bullet}(\mathfrak g)\longrightarrow N(G)$. 
Any $\theta\in \MC_n(\mathfrak g)$ determines a flat connection $\nabla_\theta:=d_{dR}+\theta$ on the trivial $G$-bundle over $\Delta^n$. 
We then define an element $e(\theta)\in G^n$ as follows: 
recalling that the vertices of $\Delta^n$ are labelled by $\{0,\dots,n\}$, we define $e(\theta)_i$, $i\in\{1,\dots,n\}$, to be the holonomy of $\nabla_\theta$ along the segment joining $i-1$ to $i$. It turns out that $e$ induces a weak equivalence $\MC_{\bullet}(\mathfrak g)\to N(G)$. 
\end{proof}
\noindent Now let $\mathfrak g^\bullet$ be a cosimplicial pronilpotent Lie algebra, so that $G^\bullet=\exp(\mathfrak g^\bullet)$ 
is a cosimplicial group. 
\begin{lem}
There is a natural weak equivalence 
$$
Tot\left(\MC_\bullet(\mathfrak g^\bullet)\right)\cong \MC_\bullet\left(\TW(\mathfrak g^\bullet)\right)\,.
$$
\end{lem}
\begin{proof}[Sketch of proof]
One knows that this is true for nilpotent Lie algebras after \cite[Theorem 4.1]{Hin}. 
The pronilpotent case comes from the fact that the functor $\MC$ commutes with limits, as well as small limits mutually commute.  
\end{proof}
\noindent The above lemma, together with the fact that 
$\pi_{\leq 1} \left(\Tot\big(N(G^\bullet)\big)\right)=Z^1(G^\bullet)$, give the result. 
\end{proof}

\subsection{Deformations of sheaves of algebras}

Back to the geometric situation, we consider the cosimplicial pronilpotent Lie algebra $\gog^\bullet=T_\cA^{+\bullet}(X)$. Its corresponding cosimplicial prounipotent Lie group $G^\bullet$ which governs the (positive) deformation theory of the sheaf $\cA$. Indeed objects of $Z^1(G^\bullet)$ are non-abelian $1$-cocycles, i.e. a collection $\left\{ \Phi_{i_0i_1}\right\}$ such that $\Phi_{i_0i_1}\in \Aut^+(\cA)\mid_{U_{i_0i_1}}$ and on $U_{i_0i_1i_2}$ we have
\[ \Phi_{i_2i_0}\circ\Phi_{i_1i_2} \circ \Phi_{i_0i_1}=\id.\]
Two such deformations are equivalent if one can be transformed to the other by an element $(\Psi_{i_0})_{i_0\in \cI}$ of the non-abelian \v Cech $0$-cochains acting on non-abelian $1$-cocycles by
\[ (\Phi_{i_0i_1}) \mapsto (\Psi_{i_1}^{-1}\circ \Phi_{i_0i_1} \circ \Psi_{i_0}).\]
Thus Theorem~\ref{thm:mc} implies the following corollary.

\begin{cor}\label{cor:mc}
There is an equivalence of groupoid
\[ \Del\big(C^*_\TW(T^+_\cA)\big) \cong Z^1\big(\exp(T_\cA^{+\bullet}(X)\big).\]
\end{cor}

\subsection{Deformation of Thom-Whitney resolutions}
Let $\left\{ \Phi_{i_0i_1}=\exp(T_{i_0i_1})\right\}$ be a non-abelian $1$-cocycle in $G^1$. Using $\Phi_{i_0i_1}$'s to glue local pieces $\cA\mid_{U_{i_0}}$ we get another sheaf of algebras which is locally isomorphic to $\cA$, but not globally. We denote this new sheaf of algebras by $\cB$. Since the sheaf $\cC^*_\TW(\cA)$ is a resolution of $\cA$, it is natural to ask how to deform this resolution to give a resolution of $\cB$.

By Corollary~\ref{cor:mc}, there is a Maurer-Cartan element $\theta$ in the differential graded Lie algebra $C^*_\TW(T^+_\cA)$ corresponding to the non-abelian $1$-cocycle $\Phi$. This element may be considered as an operator $Q$ acting on $C^*_\TW(\cA)$ thanks to the following Lemma.
\begin{lem}\label{lem:monoidal}
There is a morphism of differential graded Lie algebras
\[\alpha: C^*_\TW(T^+_\cA) \ra \Der\big( C^*_\TW(\cA), C^*_\TW(\cA)\big).\]
\end{lem}
\begin{proof}
Using that $\Der(\mathcal A)$ acts on $\mathcal A$ together with the monoidal structure of $\TW$, we get a cochain map 
$C^*_\TW(T^+_\cA)\otimes C^*_\TW(\cA) \ra C^*_\TW(\cA)$, which leads to 
$\alpha: C^*_\TW(T^+_\cA) \ra \End \big(C^*_\TW(\cA)\big)$. 
It is a straightforward computation to check that $\alpha$ is a Lie algebra map and that 
its image lies in derivations.
\end{proof}

We denote by $Q:=\alpha(\theta)$ the operator associated to $\theta$ corresponding to the non-abelian $1$-cocycle $\Phi$. Thus $Q$ is a Maurer-Cartan element of the differential graded Lie algebra $\Der(C^*_\TW(\cA))$, which implies that we can deform the differential $d$ on $C^*_\TW(\cA)$ to $d+Q$. Similarly, on the level of sheaves, we get a deformation of $\cC^*_\TW(\cA)$ whose differential will again be denoted by $d+Q$.

\begin{thm}\label{thm:taut}
There is a quasi-isomorphism
\[ \cB \ra \big(\cC^*_\TW(\cA), d+Q\big)\]
of sheaves of differential graded algebras.
\end{thm}

\begin{proof}
For each $i\in \cI$, we first show that the cohomology of $\big(\cC^*_\TW(\cA), d+Q\big)\mid_{U_i}$ is quasi-isomorphic to $\cA_i$. For this we apply Corollary~\ref{cor:mc} on $U_i$ using the covering $\left\{U_i\cap U_{i_0}\right\}_{i_0\in\cI}$. Observe that the $1$-cocycle $\Phi\mid_{U_i}$ is isomorphic to the trivial cocycle by the gauge transformation defined by the degree $0$ element $\left\{\Phi_{ii_0}\right\}_{i_0\in\cI}$. By Corollary~\ref{cor:mc} we conclude that the Maurer-Cartan section $\theta$ corresponding to $\Phi$, when restricted to $U_i$, is gauge equivalent to $0$ via a unique element $a_i\in \cC^0_\TW(T^+_\cA)(U_i)$ (corresponding to the gauge transformation $\left\{\Phi_{ii_0}\right\}_{i_0\in\cI}$) such that 
\[ e^{a_i}* (0)= \theta\mid_{U_i}.\]
Moreover the degree zero Thom-Whitney cochain $a_i$ is of the form $\big(\prod_{i_0\in \cI} T_{ii_0}, \cdots\big)$. This latter assertion follows from the proof of \cite[Theorem 4.1]{Hin}. 
Via the representation $\alpha$ we get an isomorphism of complexes of sheaves
\[ \big(\cC^*_\TW(\cA), d\big)\mid_{U_i} \stackrel{\alpha(e^{a_i})}{\longrightarrow}\big(\cC^*_\TW(\cA), d+Q\big)\mid_{U_i}.\]
The complex on the left hand side is quasi-isomorphic to $\cA_i$ via composition
\[\begin{CD}
\cA_i @>>> \cC^*(\cA)\mid_{U_{i}} @> I>> \big(\cC^*_\TW(\cA), d\big)\mid_{U_{i}}.
\end{CD}\]
Let $U_i$ and $U_j$ be two open subsets in the covering $\cU$. we consider the following diagram of maps on the intersection $U_{ij}$

\[ \begin{CD}
\cA\mid_{U_{ij}} @>>> \cC^*(\cA)\mid_{U_{ij}} @> I>> \big(\cC^*_\TW(\cA), d\big)\mid_{U_{ij}}  @> \alpha(e^{a_i})>>\big(\cC^*_\TW(\cA), d+Q\big)\mid_{U_{ij}} \\
@. @. @.     @V\id VV\\
\cA\mid_{U_{ij}} @>>> \cC^*(\cA)\mid_{U_{ij}} @<P<< \big(\cC^*_\TW(\cA), d\big)\mid_{U_{ij}}   @< \alpha(e^{-a_j})<< \big(\cC^*_\TW(\cA), d+Q\big)\mid_{U_{ij}} .
\end{CD}\]
We can use this diagram to calculate the induced gluing map on the cohomology of the sheaf $\big(\cC^*_\TW(\cA), d+Q\big)$. Using explicit formulas for $I$, $P$, and the fact that $a_i$ is of the form $\big(\prod_{i_0\in \cI} T_{ii_0}, \cdots\big)$, we conclude that this gluing map is $\Phi_{ij}$. The proof is complete.
\end{proof}

%\newpage
%\input{homotopy}
\section{Obstructions}\label{sec:homotopy}

In Section~\ref{sec:dg} we see that the resolution $T_{\tilde{X}/Y}$ of $N[-1]$ carries a differential graded Lie algebroid structure. In this section, we present a second construction of certain Lie algebroid structure on $N[-1]$ associated to a closed embedding $i:X\hookrightarrow Y$ of smooth algebraic varieties. 
In spirit this construction might be thought of as obtained from the one of Section \ref{sec:dg} by applying homological transfer technique. 
Indeed we shall end up with a strong homotopy Lie (or $L_\infty$-) algebroid structure on $N[-1]$.

\subsection{$X^{(\infty)}_Y$ versus $X^{(\infty)}_{N}$}\label{subsec:vs}

Let $k$ be either a positive integer or $\infty$, and denote by $X_Y^{(k)}$ the $k$-th infinitesimal neighbourhood of $X$ into $Y$, where by convention $X^{(\infty)}_Y:=\underrightarrow{\lim}\,X_Y^{(k)}$. Similarly, we have $X^{(\infty)}_N$, the formal neighborhood of $X$ in the normal bundle.
We borrow the notations from Subsection \ref{subsec-2.3}. Notice in particular that $\cO_X=\cO_Y/\mathcal J$ and 
$N^\chk=\mathcal J/\mathcal J^2$. 

Since the underlying topological space of the formal scheme $X^{(\infty)}_Y$ is just $X$, its structure sheaf $\cO_X^{(\infty)}$ gives a sheaf of commutative algebras on $X$ (which is not necessarily an $\cO_X$-algebra, or not even an $\cO_X$-module). Similarly, the structure sheaf of the formal scheme $X_N^{(\infty)}$ gives another sheaf of commutative algebras on $X$ which is just $\widehat{\bS}_{\cO_X}(N^\chk)$. This latter algebra is canonically an $\cO_X$-algebra via the splitting map $X_N^{(\infty)}\rightarrow N \rightarrow X$. 

We would like to compare the two sheaves of algebras $\cO_X^{(\infty)}$ and $\widehat{\bS}_{\cO_X}(N^\chk)$. The first algebra $\cO_X^{(\infty)}$ is simply the $\mathcal J$-adic completion of $\cO_Y$. Namely, there is a decreasing filtration on $\cO_Y$ given by
\[ \cO_Y={\cal J} \supset {\cal J}\supset {\cal J}^2 \supset {\cal J}^3 \supset \cdots.\]
The sheaf of algebras $\cO_X^{(\infty)}$ is, by definition, the inverse limit $\underleftarrow{\lim}\, \cO_Y/ {\cal J}^k$. On the other hand, the direct product of the associated graded components
\[ \prod_{k=0}^{\infty} {\cal J}^k/{\cal J}^{k+1}\]
inherits another commutative algebra structure from $\cO_Y$. Furthermore, by the smoothness of $X$ and $Y$, it is well-known that we have a canonical isomorphism of algebras
\begin{equation}~\label{iso}
\widehat{\bS}_{\cO_X}(N^\chk) \cong \prod_{k=0}^\infty {\cal J}^k/{\cal J}^{k+1}.
\end{equation}

\begin{lem}\label{lem-useful}
Assume that $\Phi:  \cO_X^{(\infty)}\ra \widehat{\bS}_{\cO_X}(N^\chk)$ is a filtered isomorphism of sheaves of algebras. Then the associated graded homomorphism of $\Phi$ is given by~\ref{iso} if and only if ${\rm gr}^{\leq 1} \Phi$ is the canonical isomorphism ${\rm gr}\big(\cO_X^{(1)}\big)\longrightarrow \cO_X\oplus N^\chk$.
\end{lem}

\begin{proof}
This follows from that as an algebra, $\widehat{\bS}_{\cO_X}(N^\chk)$ is generated by $\cO_X$ and $N^\chk$.
\end{proof}

\subsubsection{Affine case}

\begin{lem}~\label{local}
Assume that $Y=\Spec R$ and $X=\Spec R/I$ are two smooth affine spaces. Then there exists a filtered algebra isomorphism
\[ \Phi: \cO_X^{(\infty)}\ra \widehat{\bS}_{\cO_X}(N^\chk) \]
which, after taking the associated graded operation, agrees with the canonical isomorphism~\ref{iso}.
\end{lem}

\begin{proof}
First, we prove there exists an algebra splitting of the canonical quotient map $\cO_X^{(\infty)}=\underleftarrow{\lim}\, \cO_Y/ {\cal J}^k \ra \cO_Y/{\cal J}=\cO_X$. Under the affineness condition, this is equivalent to construct a splitting
\[ s: R/I \ra \underleftarrow{\lim}\, R/I^k.\]
For $k=1$, there is the identity morphism $s_1: R/I\ra R/I$. In general, we consider the lifting problem
\[\begin{xy} 
(0,0)*{R/I}="A"; (30,0)*{R/I^{k}}="B"; (30,30)*{R/I^{k+1}}="C";
{\ar@{->}^{s_k} "A"; "B"};
{\ar@{->}^{} "C"; "B"};
{\ar@{.>}^{s_{k+1}} "A"; "C"};
\end{xy}\]
Since the kernel of the vertical map is $I^k/I^{k+1}$ which squares to zero, by Grothendieck's formal smoothness criterion, we conclude the required lifting $s_{k+1}$ exists. Taking inverse limit $s=\underleftarrow{\lim}\, s_k$ gives the splitting we wanted.

Now let $\widehat{I}$ be the $I$-adic completion of itself in the algebra $\widehat{R}=\underleftarrow{\lim}\, R/I^k$. Via the splitting $s$, the algebra $\widehat{R}$ is endowed with a $R/I$-modules structure, which implies that $\widehat{I}$ also has a $R/I$-module structure. The canonical surjection
\[ \widehat{I}=\underleftarrow{\lim}\, I/I^k \rightarrow I/I^2 \ra 0\]
also admits a splitting as $R/I$-modules because $I/I^2$ is locally free, hence projective $R/I$-module. Let $t: I/I^2 \ra \widehat{I}$ be such a splitting. By the universal property of symmetric algebras, we get an algebra homomorphism
\[ \Psi: \widehat{\bS}_{R/I} (I/I^2) \ra \widehat{R}.\]
Passing to associated graded proves that $\Psi$ is an isomorphism. 

Finally, setting $\Phi=\Psi^{-1}$ proves the lemma. Note that by Lemma~\ref{lem-useful}, the splitting properties of $s$ and $t$ imply that the associated graded homomorphism of $\Phi$ is just~\ref{iso}.
\end{proof}

\subsubsection{Back the general case}
In the general situation of a closed embedding $i: X\hookrightarrow Y$ of smooth projective varieties. Let $U_{i_0}$ be an affine open subset of $X$. By Lemma~\ref{local}, there exists an algebra isomorphism
$$
\Phi_{i_0}: \cO_X^{(\infty)}\mid_{U_{i_0}} \longrightarrow \widehat{\bS}_{\cO_X}\big(N^\chk\big)\mid_{U_{i_0}}
$$
whose associated graded morphism is the canonical isomorphism~\ref{iso}. 

Let $U_{i_0}$ and $U_{i_1}$ be two such open subsets, and denote by $\Phi_{i_0}$ and $\Phi_{i_1}$ the associated isomorphisms. 
On the intersection $U_{i_0i_1}$, we get an automorphism
$$
\Phi_{i_0i_1}:=\Phi_{i_1} \circ \Phi_{i_0}^{-1} : 
\widehat{\bS}_{\cO_X}\big(N^\chk\big)\mid_{U_{i_0i_1}}\longrightarrow 
\widehat{\bS}_{\cO_X}\big(N^\chk\big)\mid_{U_{i_0i_1}}
$$
It moreover has the property that ${\rm gr}^{\leq1}\big(\Phi_{i_0i_1}\big)$, which is an automorphism of 
$\cO_X\oplus N^\chk$, is the identity. 

\medskip
Let $\mathcal U:=\{ U_{i_0} \} _{i_0\in \cI}$ be an open affine covering of $X$, so that we have local trivializations $\Phi_{i_0}\;(i_0\in\cI)$. 
Then the collection of isomorphisms $\{\Phi_{i_0i_1}\}_{(i_0,i_1)\in\cI^2}$ forms a non-abelian $1$-cocycle of the sheaf of groups 
$\Aut^+\left(\widehat{\bS}_{\cO_X}(N^\chk)\right)$ consisting of filtered automorphisms $\Phi$ of $\widehat{\bS}_{\cO_X}(N^\chk)$ 
such that ${\rm gr}^{\leq1}\big(\Phi\big)=id$.

\medskip

Let us denote by $\Der^+\left(\widehat{\bS}_{\cO_X}(N^\chk)\right)$ the sheaf of pronilpotent Lie algebras consisting of continuous derivations $\theta$ 
of $\widehat{\bS}_{\cO_X}(N^\chk)$ satisfying ${\rm gr}^{\leq1}\big(\theta\big)=0$. The exponential map 
$$
\exp:\Der^+\left(\widehat{\bS}_{\cO_X}(N^\chk)\right)\longrightarrow\Aut^+\left(\widehat{\bS}_{\cO_X}(N^\chk)\right)
$$
is an isomorphism.

\subsection{Homotopy Lie algebroid structure on $N[-1]$}

Let $X$ be a topological space, and let $\mathcal U$ be an open covering of $X$. For a sheaf of abelian groups $F$ on $X$, we denote by $C^*_\TW(F)$ the Thom-Whitney resolution of $F$ associated to the covering $\mathcal U$. Basic properties of the functor $C^*_\TW$ are recalled in the Section~\ref{sec:deform}. 

\begin{defi}\label{def:lie}
Let $E$ be a finitely generated locally free sheaf on a smooth algebraic variety $X$. 
An {\it $L_\infty$-algebroid} structure on $E[-1]$ is a ${\bf k}$-linear filtered derivation $Q$ of degree one on the differential graded 
algebra $\hat{\bS}_{C^*_{\TW}(\cO_X)}(C^*_{\TW}(E^\chk))$ such that $(d_{\TW}+Q)^2=0$, where $d_{\TW}$ is the original differential. \\
\indent It is called {\it minimal} if ${\rm gr}^{\leq1}(Q)=0$. 
\end{defi}

\noindent In other words, an $L_\infty$-algebroid structure is a Maurer-Cartan element $Q$ in 
$\Der\Big(\hat{\bS}_{C^*_{\TW}(\cO_X)}\big(C^*_{\TW}(E^\chk)\big)\Big)$. In the following we shall only consider minimal $L_\infty$-algebroids which are Maurer-Cartan elements in $\Der^+\Big(\hat{\bS}_{C^*_{\TW}(\cO_X)}\big(C^*_{\TW}(E^\chk)\big)\Big)$. 

Since $Q$ is a ${\bf k}$-linear derivation, it is uniquely determined by its restriction to the subspaces $\cC^*_\TW(\cO_X)$ and $\cC^*_\TW(E^\chk)$. On these subspaces $Q$ is the direct product of the following maps
\begin{align*}
a_k : \cC^*_\TW(\cO_X) & \ra \big(\cC^*_\TW(E^\chk)\big)^k \mbox{\;\; and \;\;}\\
l_k: \cC^*_\TW(E^\chk) &\ra \big(\cC^*_\TW(E^\chk)\big)^k.
\end{align*}
These maps are both of homological degree one. Just like $l_k's$ may be thought of as higher Lie brackets, the maps $a_k's$ may be viewed as the homotopy version of the anchor map.  The minimality condition implies that $a_0=0$, $l_0=l_1=0$. Thus they induce an actual Lie algebroid structure on the pair $(\mathcal O_X,E[-1])$ in ${\b D}({\bf k}_X)$~\footnote{It coincides with the one induced from $T_{\widetilde{X}/Y}$.}.The equation $(d_\TW+Q)^2=0$ implies a system of quadratic relations among $a_k$ and $l_k$'s. The first few of them are
\begin{align}~\label{compatible}
\begin{split}
&[d_\TW, a_1]=0\\
&[d_\TW,a_2]+l_2\circ a_1 =0 \\
&[d_\TW,l_2]=0\\
&[d_\TW,l_3]+(l_2\otimes\id+\id\otimes\l_2)\circ\l_2=0
\end{split}
\end{align}
The first and the third equations imply that both $a_1$ and $l_2$ are cochain morphisms. 

We also observe that there are more structure relations due to the fact that $Q$ is a derivation. This gives, for each $k\geq 1$, the equation
\begin{equation}~\label{compatible2}
 [l_k, e(f)]= e(a_{k-1}(f))
\end{equation}
where $e(-)$ is the operator given by multiplication by $-$. The following lemma is a direct consequence of the above identities.

\begin{lem}
\label{lem:dlie}
Let $E[-1]$ be a $L_\infty$ Lie algebroid over $X$ such that $a_1$ and $a_2$ vanish. Then $E[-1]$ is a Lie algebra object in the symmetric monoidal category $\D(X)$.
\end{lem}
\begin{proof}
By the equation $[l_2, e(f)]=e(a_1(f))$, we see that the vanishing of $a_1$ implies that $l_2$ is $\cC^*_\TW(\cO_X) $-linear. Thus the dual of the map $l_2$ gives rise to a morphism  
\[ \Lambda^2(E[-1]) \ra E[-1]\]
in $\D(X)$. Next we prove that $a_2=0$ implies that this morphism satisfies the Jacobi identity. Indeed, the vanishing of $a_2$ implies that 
\[ [l_3,e(f)]=e(a_2(f))=0,\]
which proves that $l_3$ is also $\cC^*_\TW(\cO_X) $-linear. By the fourth equation in~\ref{compatible}, this furthermore implies the Jacobi identity .
\end{proof}

Later we will be able to give a geometric interpretation for the vanishing of $a_1$ and $a_2$ in the case of $N[-1]$ associated to $i:X\hookrightarrow Y$, see Proposition~\ref{prop:obs1}.
\begin{prop}\label{coro:main}
Let $i: X\hookrightarrow Y$ be a closed embedding of smooth algebraic varieties. 
Then there is a minimal $L_\infty$-algebroid structure on $N[-1]$ such that its Chevalley-Eilenberg algebras is quasi-isomorphic to $\cO^{(\infty)}_X$.
\end{prop}
\begin{proof}
We apply Theorem~\ref{thm:taut} in the case $\cA=\hat{\bS}_{\cO_X} N^\chk$, $T_\cA^+=\Der^+\big(\hat{\bS}_{\cO_X} N^\chk\big)$, $\cB=\cO^{(\infty)}_X$, and the $1$-cocycle $\Phi$ constructed in the previous subsection.
\end{proof}

\begin{rem}
The cocycles $a_1$ associated to the $L_\infty$-algebroid structure on $N[-1]$ defines a cohomology class
\[[a_1]\in \Ext^1(\Omega_X, N^\chk)=\Ext^1(N,T_X).\]
Geometrically, the class $[a_1]$ corresponds to the extension class of the short exact sequence
\[ 0\ra T_X \ra T_Y\mid_X \ra N \ra 0.\]
Assuming that $a_1=0$, by equation~\ref{compatible2} we deduce that $l_2$ defines a class
\[ [l_2]\in \Ext^1(N^\chk, S^2 N^\chk)=\Ext^1(S^2N, N).\]
This class is given by the extension class of the sequence
\[ 0\ra {\cal J}^2/{\cal J}^3 \ra {\cal J}/{\cal J}^3 \ra {\cal J}/{\cal J}^2 \ra 0.\]
Note that since $a_1=0$, the tangent sequence above splits, which is equivalent to an algebra splitting $s_1: \cO_X \ra \cO_X^{(1)}$ of the first order neighborhood. Via this splitting, the sheaf ${\cal J}/{\cal J}^3$, {\sl a priori} an $\cO_X^{(1)}$-module, maybe viewed as an $\cO_X$-module. Alternatively, fixing a splitting of the tangent sequence, we get an isomorphism
\[ T_Y\mid_X \cong T_X\oplus N.\]
The extension class $[l_2]$ is simply the normal component of the Atiyah class of $T_Y$ pulled back to $X$.
\end{rem}

\subsection{Obstructions}
The structure maps of the $L_\infty$ algebroid $N[-1]$ can be used to describe various cohomology obstructions when comparing $X^{(k)}_N$ with $X^{(k)}_Y$. Let us set
\[\cF:=\Big(\hat{\bS}_{\cC^*_\TW(\cO_X)}\big(\cC^*_{\TW}(N^\chk)\big), d_\TW+Q\Big)\]
the Chevalley-Eilenberg algebra of $N[-1]$ defined in Proposition~\ref{coro:main}. 
Observe that since $\rm gr(\Phi_{i_0i_1})=0$, the operator $Q$ preserves the augmentation ideal $\cF^+$ consisting of symmetric tensors of degree at least $1$. Thus $\cF^+$ is a differential ideal in $\cF$. We define a quotient differential graded algebra $\cF^{(k)}:= \cF/(\cF^+)^{k+1}$ for each $k\geq 0$. By Theorem~\ref{thm:taut} we have a quasi-isomorphism
\[ \cO_{X/Y}^{(k)} \ra \cF^{(k)}.\]

\subsubsection{Splittings of $X\hookrightarrow X^{(k)}_Y$}

Let $s_k: X_Y^{(k)} \ra X$ be a splitting of the embedding $X\hookrightarrow X^{(k)}_Y$. We again use $s_k: \cO_X \ra \cO^{(k)}_X$ to denote the corresponding morphism on rings. In this case when forming the cocycle $\Phi$ we can further require that $\Phi_{i_0}$ to fill the following commutative diagram
\[\begin{CD}
\cO_X^{(k)} \mid_{U_{i_0}} @> \Phi^{(k)}_{i_0} >> \bS^{\leq k}_{\cO_X} N^\chk \mid_{U_{i_0}} \\
@A s_k AA    @AAA\\
\cO_X\mid_{U_{i_0}} @=  \cO_X\mid_{U_{i_0}}
\end{CD}\]
where the right vertical arrow is the obvious inclusion map. This restriction implies that the corresponding derivations $T_{i_0i_1}$ is contained in the subset of $\Der^+\big( \hat{\bS}_{\cO_X} N^\chk\big)$ consisting of derivations such that the component
\[ \cO_X \ra \bS^i N^\chk \]
vanishes for all $ 0\leq i\leq k$. Observe that this is a Lie subalgebra due to the minimality condition. We shall call such a $1$-cocycle $\Phi$ compatible with $s_k$.

If $\Phi$ is a $1$-cocycle compatible with $s_k$, then we have a corresponding homotopy $L_\infty$ algebra structure on $N[-1]$. Recall from Lemma~\ref{lem:monoidal} that the operator $Q$ is the image of a Maurer-Cartan element $\theta$ under the morphism
\[ \alpha: C^*_\TW\Big(\Der\big( \hat{\bS}_{\cO_X} N^\chk\big)\Big) \ra \Der\big( C^*_\TW(\hat{\bS}_{\cO_X} N^\chk)\big)\]
defined in \emph{loc. cit.}. Hence the structure maps $a_k$ and $l_k$, being the component maps of $Q$, also lies in the image of $\alpha$. If we denote by $\theta_{0,k}$ and $\theta_{1,k}$ the components of $\theta$ such that
\begin{align*}
\theta^{0,k}&\in C^*_\TW\Big(\Der\big( \cO_X, (N^\chk)^k\big)\Big), \mbox{\;\; and \;} \\
\theta^{1,k}& \in C^*_\TW\Big(\Hom_{k_X}\big(N^\chk, (N^\chk)^k\big)\Big),
\end{align*}
then we have
\[ a_k=\alpha(\theta^{0,k}), \mbox{\;\;\; and \;\;} l_k=\alpha(\theta^{1,k}).\]

\begin{lem}
\label{lem:obs1}
Let $s_k: \cO_X \ra \cO_X^{(k)}$ be a splitting, and let $\Phi$ be a $1$-cocycle compatible with $s_k$. Let $a_i=\alpha(\theta^{0,i})$ be the associated structure maps in the $L_\infty$ algebroid structure on $N[-1]$. Then we have $\theta^{0,i}=0$ and $a_i=0$ for $0\leq i\leq k$. Moreover both $a_{k+1}$ and $\theta^{0,k+1}$ are cocycles.
\end{lem}

\begin{proof} 
The vanishing of $\theta^{0,i}$ and $a_i$ is clear from the definition. Now let us show that $\theta^{0,k+1}$ is a cocycle. Indeed since $\theta$ satisfies the Maurer-Cartan equation, we have
\[ [d_\TW, \theta^{0,k+1}]+\sum_{i+j=k+2} \theta^{1,j}\circ \theta^{0,i}=0.\]
Due to minimality condition the sum above is over $j\geq 2$, which in particular implies that $i\leq k$. Hence the sum $\sum_{i+j=k+2} \theta^{1,j}\circ \theta^{0,i}$ vanishes, and so we get
\[ [d_\TW,\theta^{0,k+1}]=0.\]
This proves that $\theta^{0,k+1}$ is a cocycle. Since $a_{k+1}$ is the image of $\theta^{0,k+1}$ under the cochain map $\alpha$, it is also a cocycle.
\end{proof}

Thus the element $\theta^{0,k+1}$ is a degree one cocycle in the complex $C^*_\TW\Big(\Der\big( \cO_X, (N^\chk)^k\big)\Big)$. Its class $[\theta^{0,k+1}]$ is then a class in $\Ext^1(\bS^{k+1}N, T_X)$. Similarly the class $[a_{k+1}]$ is in the first cohomology group $H^1\big(\Hom(\cC^*_\TW(\bS^{k+1}_{\cO_X} N), T_{\cC^*_\TW(\cO_X)})\big)$. It is plausible the two cohomology groups are in fact isomorphic, but the authors do not know a proof of this claim.

\begin{prop}
\label{prop:obs1}
Assume that we are in the same setup as Lemma~\ref{lem:obs1}, then the following are equivalent:
\begin{itemize}
\item[$(A)$] there exists a splitting $s_{k+1}: \cO_X\ra \cO_X^{(k+1)}$ lifting $s_k$;
\item[$(B)$] the cohomology class $[\theta^{0,k+1}]\in \Ext^1(\bS^{k+1}N, T_X)$ vanishes;
\item[$(C)$] the cohomology class $[a_{k+1}] \in H^1\big(\Hom(\cC^*_\TW(\bS^{k+1}_{\cO_X} N), T_{\cC^*_\TW(\cO_X)})\big)$ vanishes.
\end{itemize}
\end{prop}

\begin{proof}
$\bullet~(A)\Rightarrow(B)$: Assuming $(A)$, we can choose another 1-cocycle $\Phi'$ to be compatible with $s_{k+1}$, which in particular 
implies that it is compatible with $s_k$. Denote by $\theta'$ the corresponding Maurer-Cartan element. Since different choices of $\Phi$ give 
rise to gauge equivalent Maurer-Cartan element by Corollary~\ref{cor:mc}, there exist a degree zero element $\eta$ in the differential graded 
Lie algebra $\cC^*_\TW(\gog)$ such that
\[ e^\eta*\theta' = \theta\]
where $\gog$ is the Lie subalgebra of $\Der^+(\hat{\bS}_{\cO_X} N^\chk)$ whose $(0,i)$ components vanishes for $0\leq i\leq k$. 
Writing out the above equation in the $(0,k+1)$-component implies that 
\[ [d_\TW, \eta^{0,k+1}] = \theta^{(0,k+1)}.\]
This proves $(B)$.

$\bullet~(B)\Rightarrow(C)$: It suffices to note that on the chain level we have $a_{k+1}=\alpha(\theta^{(0,k+1)})$.

$\bullet~(C)\Rightarrow(A)$: Assume that $a_{k+1}=[d_\TW, h]$ for some degree zero 
$h\in \Hom(\cC^*_\TW(\bS^{k+1}_{\cO_X} N), T_{\cC^*_\TW(\cO_X)})$. We define a morphism
$$
\tilde{s}_{k+1}: \cF^{(0)} \ra \cF^{(k+1)}
$$
by $s_{k+1}(f):= f+ h(f)$. Recall that $\cF^{(k)}$ is the $k$-truncation of 
$\cF:=\Big(\hat{\bS}_{\cC^*_\TW(\cO_X)}\big(\cC^*_{\TW}(N^\chk)\big), d_\TW+Q\Big)$. 
One checks that $\tilde{s}_{k+1}$ is a morphism of differential graded algebras. Moreover the composition
$$\cF^{(0)} \stackrel{\tilde{s}_{k+1}}{\ra} \cF^{(k+1)} \twoheadrightarrow \cF^{(k)}
$$
is simply $f\mapsto f$. Thus the induced map on cohomology gives a map of algebras
$$
s_{k+1}: \cO_X \ra \cO^{(k+1)}_X
$$
that lifts the splitting map $s_k$. 
\end{proof}

\begin{cor}
Let $i:X\ra Y$ be an embedding of smooth algebraic varieties. Assume that there exists a splitting $s_2$ of the natural morphism $X\hookrightarrow X^{(2)}$. Then $N[-1]$ is a Lie algebra object in the symmetric monoidal category $\D(X)$.
\end{cor}

\begin{proof}
This is a consequence of Proposition~\ref{prop:obs1} and Lemma~\ref{lem:dlie}.
\end{proof}
\begin{rem}
The Lie structure on $N[-1]$ depends on the choice of $s_2$.
\end{rem}
%IT SEEMS, AS DAMIEN ALSO POINTED OUT, THE MAP $\alpha$ MAY ITSELF BE AN ISOMORPHISM.

\subsubsection{When is $X^{(k)}_N$ isomorphic to $X^{(k)}_Y$}

Let us assume that there is an isomorphism $t_{k-1}: X^{(k-1)}_N\cong X^{(k-1)}_Y$. We would like to understand when $t_{k-1}$ lifts to an isomorphism between $X^{(k)}_N$ and $X^{(k)}_Y$.

First note that the isomorphism $t_{k-1}$ induces a splitting $s_{k-1}: X^{(k-1)}_Y \ra X$. Hence we can use Proposition~\ref{prop:obs1} to analyze the lifting of $s_{k-1}$ to a splitting $s_k$, which should necessarily exist if $t_{k-1}$ lifts. Thus in the following we assume that there is a splitting $s_k: X^{(k)}_Y \ra X$ compatible with $t_{k-1}$ in the sense that the induced splitting from $t_{k-1}$ agrees with that of $s_k$.

Given a compatible pair $(t_{k-1}, s_k)$ as above, we can require the isomorphisms $\Phi_{i_0}$ to be compatible with $(t_{k-1}, s_k)$ in the sense that it is compatible with $s_k$, and that there is a commutative diagram.
\[ \begin{CD}
\cO_X^{(k-1)} \mid_{U_{i_0}} @> \Phi^{(k-1)}_{i_0} >> \bS^{\leq k-1}_{\cO_X} N^\chk \mid_{U_{i_0}} \\
@A t_{k-1} AA    @|\\
\bS^{\leq k-1}_{\cO_X} N^\chk \mid_{U_{i_0}}  @= \bS^{\leq k-1}_{\cO_X} N^\chk \mid_{U_{i_0}} 
\end{CD}\]
The corresponding $1$-cocycle $\Phi_{i_0i_1}$ is then called compatible with $(t_{k-1}, s_k)$.

\begin{lem}
\label{lem:obs2}
Let $(t_{k-1}, s_k)$ be a compatible pair, and let $\Phi$ be a compatible $1$-cocycle as described above. Then we have $a_i=0$ for all $0\leq i\leq k$, and $l_i=0$ for all $0\leq i\leq k-1$. Moreover the morphism $l_k$ is a cocycle.
\end{lem}

\begin{proof}
The proof is similar to that of Lemma~\ref{lem:obs1}. Indeed the vanishing of $a_i$ and $l_i$ follow from definition while the identity $[d_\TW,l_k]=0$ follows from the Maurer-Cartan equation $(d_\TW+Q)=0$ and the vanishing of lower degree $a_i$'s and $l_i$'s.
\end{proof}

The morphism $l_k$ is a degree one cocycle in the complex $\Hom_{C^*_\TW(\cO_X)}\big(C^*_\TW(N^\chk), C^*_\TW(\bS^k_{\cO_X}N^\chk)\big)$. Since the functor $\TW$ is monoidal this complex computes $\Ext^*_X(N^k, N)$.

\begin{prop}
\label{prop:obs2}
Let the setup be the same as in Lemma~\ref{lem:obs2}. Then the following are equivalent:
\begin{itemize}
\item[$(A)$] the class $[l_k]\in \Ext^1_X(N^k,N)$ vanishes;
\item[$(B)$] the isomorphism $t_{k-1}$ lifts to an isomorphism $t_k: \bS^{\leq k}_{\cO_X} N^\chk \ra \cO^{(k)}_X$.
\end{itemize}
\end{prop}

\begin{proof}
$\bullet~(A)\Rightarrow(B)$: Assume that $l_k=[d_\TW, h]$ for some degree zero morphism $h$ in the morphism complex 
$\Hom_{C^*_\TW(\cO_X)}\big(C^*_\TW(N^\chk), C^*_\TW(\bS^k_{\cO_X}N^\chk)\big)$. We can then define a morphism of algebras 
$\tilde{t}_k: C^*_\TW(\bS^k_{\cO_X}N^\chk) \ra \cF^{(k)}$ by formula $\id+h$. Notice that this is a morphism of algebras since $h$ has image in 
$\bS^k$. That it also commutes with differential is a computation using Lemma~\ref{lem:obs2} and the identity $[d_\TW,h]=l_k$. 
Thus the induced map on cohomology defines a required lifting. Note that it lifts $t_{k-1}$ since $\id+h$ modulo $\bS^k$ is just $\id$.

$\bullet~(B)\Rightarrow(A)$: Assuming that there exists such a lifting, then we can choose a different cocycle $\Phi'$ such that the 
corresponding structure maps satisfy $a_i=0$ and $l_i=0$ for all $0\leq i\leq k$. Same as in the proof of Proposition~\ref{prop:obs1}, 
we can use the fact that $\Phi$ and $\Phi'$ are gauge equivalent to show that $l_k$ is exact.
\end{proof}

\begin{rem}
Similar obstruction classes as in Proposition~\ref{prop:obs1} and Proposition~\ref{prop:obs2} were discussed in~\cite[Proposition 2.2]{ABT} and~\cite[Corollary 3.4 \& Corollary 3.6]{ABT} using complex analytic methods. The existence of such classes can be formulated using the language of gerbes and stacks, see for example~\cite[Section 4]{Kap}. The observation that all these obstruction classes come from an $L_\infty$-algebroid structure seems to be new.
\end{rem}

\end{document}